\documentclass{amsart}
\usepackage{amssymb}
\usepackage{stmaryrd} 
\usepackage{amsmath} 
\usepackage{amscd}
\usepackage{amsthm}
\usepackage{amsbsy}
\usepackage[margin=1in]{geometry}
\usepackage{commath, longtable}
\usepackage{comment, enumerate}
\usepackage{geometry}
\usepackage[matrix,arrow]{xy}
\usepackage{hyperref}
\usepackage{mathrsfs}
\usepackage{color}
\usepackage{float}
\usepackage{mathtools,caption}
\usepackage[table,dvipsnames]{xcolor}
\usepackage{tikz-cd}
\usepackage{longtable}
\usepackage[utf8]{inputenc}
\usepackage[OT2,T1]{fontenc}
\usepackage[normalem]{ulem}
\usepackage{hyperref}
\usepackage[customcolors]{hf-tikz}
\usepackage{enumitem}
\newlist{primenumerate}{enumerate}{1}
\setlist[primenumerate,1]{label={\arabic*$'$}}
\hypersetup{
 colorlinks=true,
 linkcolor=DarkOrchid,
 filecolor=blue,
 citecolor=olive,
 urlcolor=orange,
 pdftitle={GHKL project},
 }
\usepackage{booktabs}

\DeclareSymbolFont{cyrletters}{OT2}{wncyr}{m}{n}
\DeclareMathSymbol{\Sha}{\mathalpha}{cyrletters}{"58}

\newtheorem{theorem}{Theorem}[section]
\newtheorem{lemma}[theorem]{Lemma}

\newtheorem{proposition}[theorem]{Proposition}
\newtheorem{corollary}[theorem]{Corollary}
\newtheorem{definition}[theorem]{Definition}

\newtheorem*{theorem*}{Theorem}
\numberwithin{equation}{section}

\theoremstyle{remark}
\newtheorem{remark}[theorem]{Remark}

\newcommand{\new}{\mathrm{new}}
\newcommand{\val}{\mathrm{val}}

\newcommand{\SL}{\mathrm{SL}}
\newcommand{\NSq}{\mathrm{NSq}}
\newcommand{\Sq}{\mathrm{Sq}}

\newcommand{\Z}{\mathbb{Z}}

\makeatletter
\theoremstyle{plain} 
\newtheorem*{intr@thm}{\intr@thmname}

\newtheorem*{c@njecture}{\conjn@name}
\newcommand{\myl@bel}[2]{
 \protected@write \@auxout {}{\string \newlabel {#1}{{#2}{\thepage}{#2}{#1}{}} }
 \hypertarget{#1}{}
 } 

\makeatother

\makeatletter
\newcommand{\mylabel}[2]{#2\def\@currentlabel{#2}\label{#1}}
\makeatother

\title[]{Murmurations of Modular Forms and $p$-power Coefficients}

\author[D.~Kundu]{Debanjana Kundu}
\address[DK]{University of Texas - Rio Grande Valley, USA}
\email{debanjana.kundu@utrgv.edu \textrm{or} tuli93@gmail.com}

\author[K.~M\"uller]{Katharina M\"uller}
\address[Müller]{Institut für Theoretische Informatik, Mathematik und Operations Research, Universität der Bundeswehr München, Werner-Heisenberg-Weg 39, 85577 Neubiberg, Germany}
\email{katharina.mueller@unibw.de}

\keywords{Murmurations, prime power coefficients of a modular form}
\subjclass[2020]{Primary: 11F11, 11F30}

\begin{document}

\begin{abstract}
We extend the work of N.~Zubrilina on murmuration of modular forms to the case when prime-indexed coefficients are replaced by squares of primes.
Our key observation is that the shape of the murmuration density is the same.
\end{abstract}

\maketitle

\section{Introduction}

The authors of \cite{hlop22} were the first to notice the oscillating pattern of the average value of the $p$-th Dirichlet coefficients of fixed rank elliptic curves for a prime $p$ in a fixed conductor range.
This kick-started the study of what is now known as `murmurations' of elliptic curves.
This pattern was then detected in more general families of arithmetic $L$-functions, such as those associated to weight $k$ holomorphic modular cusp forms for $\Gamma_0(N)$ with conductor in a geometric interval range $[M,cM]$ and a fixed root number\footnote{\href{https://math.mit.edu/~drew/RubinsteinSarnakLetter.pdf}{Sutherland's letter to Rubinstein and Sarnak}}. 
A.~Sutherland also made a striking observation that the average of $a_f(P)$ over this family for a single prime $P \sim M$ converges as a continuous looking function of $P/M$.
More recently, N.~Zubrilina established a case of the correlation phenomenon between Fourier coefficients of families of modular forms and their root numbers in \cite{Zub23}.
The purpose of this article is to extend the main result of Zubrilina to the case of prime power coefficients of modular forms.

The first main result proven by Zubrilina for weight 2 modular forms is the following:

\begin{theorem*}
Let $H^{\new}(N)$ be a Hecke basis for trivial character weight $2$ cusp new forms for $\Gamma_0(N)$ with $f \in H^{\new}(N)$ normalized to have lead coefficient 1. Let $\epsilon(f)$ be the root number of $f$, let $a_f(p)$ be the $p$-th Fourier coefficient of $f$, and let $\lambda_f(p):=a_f(p)/\sqrt{p}$.
Let $X,Y$, and $P$ be parameters going to infinity with $P$ prime; assume further that $Y=(1+o(1))X^{1-\delta_2}$ and $P\ll X^{1+\delta_1}$ for some $\delta_1, \delta_2 >0$ with $2 \delta_1 < \delta_2 < 1$.
Let $y:=P/X$.
Then 
{\allowdisplaybreaks
\begin{align*}
\frac{\displaystyle\sum_{\substack{N\in[X,X+Y] \\ \textrm{square-free}}} \sum_{f\in H^{\new}(N)} \lambda_f(P)\sqrt{P} \epsilon(f)}{\displaystyle\sum_{\substack{N\in[X,X+Y] \\ \textrm{square-free}}} \sum_{f\in H^{\new}(N)} 1 } = &\frac{12}{\pi \prod_{p} \left( \frac{1}{p(p+1)}\right)} \left(A \sqrt{y} + B \sum_{1\leq r \leq 2\sqrt{y}} C(r)\left(\sqrt{4y-r^2}\right) - \pi y\right) \\
& + O_{\epsilon}\left( X^{-\delta'+ \epsilon}+ \frac{1}{P}\right).
\end{align*}
}
where 
{\allowdisplaybreaks
\begin{align*}
A & = \prod_p\left(1+\frac{p}{(p+1)^2(p-1)}\right)\\
B & = \prod_p\left(\frac{p^4 - 2p^2 - p +1}{(p^2-1)^2}\right)\\
C(r) & = \prod_{p\mid r}\left( 1+ \frac{p^2}{p^4 - 2p^2 - p +1}\right).
\end{align*}}
In particular, for any $\delta_1 < \frac{2}{9}$, one can find $\delta_2$ for which $\delta' > 0$.
\end{theorem*}

As explained in \emph{op. cit.} the formula obtained above (referred to as the murmuration density) comes from applying the Eichler--Selberg trace formula to the composition of Hecke and Atkin--Lehner operators.
Using this technique allows for interpreting the sum in terms of certain class numbers, the averages of which in short intervals can then be handled via the class number formula.

The main result of this paper is Theorem~\ref{main theorem} which we record below for the convenience of the reader.

\begin{theorem*}
Let $H^{\new}(N)$ be a Hecke basis for trivial character for weight 2 cusp forms for $\Gamma_0(N)$ with $f\in H^{\new}(N)$ normalized to have leading coefficient 1.
Let $\epsilon(f)$ be the root number of $f$ and let $a_f(p)$ be the $p$-th Fourier coefficient of $f$, and set $\lambda_f(p) = \frac{a_f(p)}{\sqrt{p}}$.
Let $P$ be a prime, and suppose that the parameters $P$, $X$, and $Y$ go to infinity.
Set $Y= (1+o(1))X^{1-\delta}$ and $P^2\ll X^{1+\delta_2}$
where $0<\delta_2 < \delta < \frac{9}{11}$ and $\frac{\delta}{9} + \frac{\delta_2}{2} < \frac19$.
Set $\delta'=\delta_2 - \frac{\delta}{2}$.
Then for $A, B$, and $C(r)$ as defined before and writing $y=\frac{P^2}{X}$,
{\allowdisplaybreaks
\begin{align*}
\frac{\displaystyle\sideset{}{'}\sum_{\substack{N\in [X, X+Y]}}\sum_{f\in H^{\new}(N)} P \lambda_{f}(P^2)\epsilon(f)}{\displaystyle\sideset{}{'}\sum_{\substack{N\in [X, X+Y]}}\sum_{f\in H^{\new}(N)} 1} = & \frac{12}{\pi \prod_p\left( 1 - \frac{1}{p(p+1)}\right)}\left(A\sqrt{y} + B\sum_{r\leq \sqrt{y}} C(r)\left( \sqrt{4y-r^2}\right) - \pi y\right) \\
& + O_\epsilon\left( X^{\delta' + \epsilon} +\frac{1}{P}\right).
\end{align*}}
Here, the notation $\sideset{}{'}\sum$ indicates that the sum is over square-free $N$.
\end{theorem*}

Our key observation is that qualitatively both results are the same.
Our approach in proving the result also uses the trace formula (as explained above). 
However, as we understand it, our result \emph{can not} be obtained as a corollary of Zubrilina's result.
While Zubrilina's result establishes a relation between the $P$-th Fourier-coefficient of a family of modular forms and their root numbers, our result proves a relation between the $P^2$-th Fourier-coefficient of modular forms and their roots number.

\subsection*{Other Results}
Our main result mimics \cite[Theorem~1]{Zub23} qualitatively.
It is immediate that other results proven in \cite{Zub23} have the obvious analogue in our setting. 

We prove our result for weight $2$ modular forms but we remark that the result can be proven for any higher weight modular form.
We expect similar shape of results if $P^2$ is replaced by a higher power of $P$.
We refrain from proving the result in more generality to keep the notation less cumbersome.
Also, since our result proves that the shape of the average size of the $P^2$-coefficients is qualitatively the same shape of the average size of the $P$-coefficients, the general result on murmuration density can be easily read from \cite{Zub23}.

It will also be straightforward to extend this idea to the case where $P$ is replaced by products of two (or more) primes.
The main difficulty will arise form the fact that the trace formula result mentioned in Theorem~\ref{Skoruppa-Zagier} will have more terms so the calculations will be more tedious.

\subsection*{Outlook}
There are examples where one knows that the murmuration density is qualitatively different, see \cite{BBLL23}.
For this family, the answer is given in terms of a measure and turns out to be exactly identical (up to constants) to the case of Maass forms (we learned this from personal communication with Zubrilina).

One question of potential interest is what feature of the family determines the qualitative shape of the answer.
Given the philosophical connection to one-level density, a na{\"i}ve guess at an answer would be the symmetry type of the family, but this requires further investigation.
At the moment no ansatz to address this question in some general framework is known but examples could provide evidence for the above guess.

It might be interesting to understand the murmurations for the family of symmetric squares of modular forms.
Note that for the $L$-function $L(s, \operatorname{sym}^2 f)$, the root number is always +1 and conductor is $N^2$ (rather than $N$).
Since this is a symplectic family, we expect the murmuration function to reflect a different transition in the
one-level density.
The numerical work of Sutherland also suggests that we expect degree 3 $L$-functions to have a different normalization, as in they appear most regular as a function of $(\frac{P}{N})^{1/3}$.
Studying the case of symmetric squares will be more delicate.

\subsection*{Organization}
Including the introduction, this article has seven sections.
In Section~\ref{Sec: Setup} we describe the set-up of the main problem and state our main results.
For the purpose of our calculations, we need to analyze two main terms, this is done is Sections~\ref{Section: Main Term Calculation 1} and Section~\ref{Section: Main Term Calculation 2}.
The calculation for the remainder term(s) is carried out in Section~\ref{Section: Remainder}.
Finally, in Section~\ref{section: the remaining terms} we handle the terms which are outside the the range of terms addressed in earlier sections, those involving levels $N$ when $P\mid N$, and those involving the $P$ term in the trace formula.
Section~\ref{Sec: Arithmetic Functions} is dedicated to recording technical lemmas involving arithmetic functions.

\subsection*{Acknowledgements}
DK thanks Nina Zubrilina for her wonderful talk at UTRGV on this topic.
We are grateful to the referee for their constructive feedback which helped improve the exposition of this paper.
DK acknowledges the support of an AMS--Simons early career travel grant.

\section{Trace Formula Set-up}
\label{Sec: Setup}

\subsection{Set-up}
Fix a square-free positive integer $N$. 
Set $H^{\new}(N)$ to denote a basis of the space $S^{\new}(N)$ of weight $2$ Hecke cusp newforms for $\Gamma_0(N)$, and let $a_f(P) = \lambda_f(P)\sqrt{P}$ denote the eigenvalue under the $P$-th Hecke operator $T_p$ of $f \in H^{\new}(N)$.
Let $\epsilon(f)$ denote the root number of $f$ and recall that $-\epsilon(f)$ is equal to the eigenvalue of $f$ under the Atkin--Lehner involution $W_N$.

In contrast to previous considerations of murmuration behaviour we are interested in the Fourier-coefficient $\lambda_f(P^2)$.
When $P\nmid N$, we know that 
\[
 \lambda_f(P^2) = \lambda_f(P)^2 -1.
\]

Consider the operator $(-1)T_{P^2}\circ W_N$ on $S^{\new}(N)$; its trace is 
\[
\sum_{f\in H^{\new}(N)} a_f(P^2)\epsilon(f) = \sum_{f\in H^{\new}(N)} P \lambda_f(P^2)\epsilon(f). 
\]
More generally, if we consider $P^k$ Fourier coefficients, we apply the operator $(-1)T_{P^k} \circ W_N$ on $S^{\new}(N)$ and its trace is
\[
\sum_{f\in H^{\new}(N)} a_f(P^k)\epsilon(f) = \sum_{f\in H^{\new}(N)} P^{\frac{k}{2}} \lambda_f(P^k)\epsilon(f). 
\]
On the other hand, when $P\mid N$, 
we have the equality 
\[
\lambda_f(P^\nu) = \left(\frac{a_f(P)}{\sqrt{P}}\right)^\nu \text{ for all } \nu\geq 0.
\]
Recall that $a_f(P)=\pm1$ so, for $\nu$ even  we have 
\[
\lambda_f(P^{2\nu'}) = P^{-\nu'} \text{ for all } \nu'\geq 0.
\]
In particular, when $\nu'=1$, we have $\lambda_f(P^2) = \frac{1}{P}$.

\subsection{Trace Formula}
We record a result of Skoruppa--Zagier \cite[Section~2]{skoruppa1987jacobi} that will be useful for us.
The following result is a special case of the formula in equation (7) of \emph{op.~cit.}
The result there is more general but since we restrict ourselves to $N$ square-free, the expression simplifies in our setting to the following

\begin{theorem}
\label{Skoruppa-Zagier}
With notation introduced above, in the case that $P\nmid N$
for $P$-power coefficients
\[
\sum_{f\in H^{\new}(N)} a_f(P^k)\epsilon(f) = \sum_{f\in H^{\new}(N)} P^{\frac{k}{2}} \lambda_f(P^k)\epsilon(f) = \frac{H_1(-4P^kN)} {2} + \sum_{0<r\leq 2 \frac{P^{\frac{k}{2}}}{\sqrt{N}}} H_1(r^2N^2 - 4P^{k}N) - \left(\sum_{i=0}^k P^i\right).     
\]
In particular, when $k=2$
\[
\sum_{f\in H^{\new}(N)} a_f(P^2)\epsilon(f) = \sum_{f\in H^{\new}(N)} P \lambda_f(P^2)\epsilon(f) = \frac{H_1(-4P^2N)} {2} + \sum_{0<r\leq 2P/\sqrt{N}} H_1(r^2N^2 - 4P^2N) - (1+P+P^2).     
\]

\end{theorem}

\begin{remark}
In the notation of \cite{skoruppa1987jacobi}, the quantity $n_1$ takes the value $N$ and $n_2 =1$.
A priori, the second term should have a coefficient $U_0(r\sqrt{N}/2P^{k/2})$ but we note that $U_0(x)=1$.
The final term in the expression is $\sigma_1(P^{k}) = \sum_{d\mid P^{k}} d = 1 + P + P^2 + \ldots + P^k$.

In the above formula, the Hurwitz class number $H_1(-d)$ is the number of equivalence classes with respect to $\SL_2(\Z)$ of positive-definite binary quadratic forms of discriminant $-d$ weighed by the number of automorphisms (i.e., with forms corresponding to multiples of $x^2 +y^2$ and $x^2 +xy+y^2$ counted with multiplicities $\frac{1}{2}$ and $\frac{1}{3}$, accordingly).
In other words, $H_1$ can be expressed in terms of the Gauss class number $h$ via: 
\[
H_1(-d) = \sum_{f\in \Z^{>0}, f^2\mid d} h(-d/f^2) +O(1),
\]
with the error term disappearing if $d$ is not of the form 3 or 4 times a square.
\end{remark}

Using the relationship between the Hurwitz class number $H_1(\cdot)$ and Gauss class number $h(\cdot)$ explained above, we see that
\[
H_1(-4P^kN) = \begin{cases}
\displaystyle{\sum_{j=0}^{k/2} h(-4P^{2(\frac{k}{2} - j)}N) + \sum_{j=0}^{k/2} h(-P^{2(\frac{k}{2} - j)}N)} + O(1) & \textrm{ when }k=\textrm{ even}\\
\displaystyle{\sum_{j=0}^{(k-1)/2} h(-4P^{2(\frac{k-1}{2} - j)}PN) + \sum_{j=0}^{(k-1)/2} h(-P^{2(\frac{k-1}{2} - j)}PN)} + O(1) & \textrm{ when }k=\textrm{ odd}.\end{cases}
\]
Throughout this article we assume that $P\neq 2$.
Since we assume that $\gcd(P,N)=1$ and that $N$ is square-free, the square factors of $4P^{k}N$ are $1, 4, P^2, 4P^2, \ldots, P^{\lfloor k/2 \rfloor}, 4P^{\lfloor k/2 \rfloor}$. 
The above expression simplifies to the following formula when $k=2$ (which is what is required for our main result)
\[
H_1(-4P^2N) = h(-4P^2N) + h(-P^2N) + h(-4N) + h(-N) + O(1).
\]
If $q$ is any prime and $r\geq 1$, we note that $q^2\mid N(r^2 N -4P^{k})$ is satisfied if $q^2\mid (r^2 N -4P^{k})$ or if $q\mid N$ and $q\mid 4P^{k}$.
The latter condition is satisfied precisely when $q=2$ (i.e., $N$ is even).
Writing $N=2N'$ (where $N'$ must now be odd), if $4d^2 \mid (r^2 N^2 -4P^{k}N)$, we can note that
\[
\frac{r^2 N^2 -4P^{k}N}{4d^2} = \frac{r^2 N'^2 -2P^{k}N'}{d^2}
\]
is not a square modulo 4 and the corresponding class number vanishes.
Indeed, we remind the reader that by Dirchlet's class number formula, for $d>4$, the Gauss class number $h(-d)=0$ whenever $-d \equiv 2 \text{ or }3\pmod{4}$.
It will therefore be enough to consider those square-divisors which satisfy $d^2\mid (r^2N - 4P^{k})$.  
In the general case, the trace formula becomes 
\begin{align*}
\sum_{f\in H^{\new}(N)} P \lambda_f(P^k)\epsilon(f) & =  \displaystyle{\sum_{j=0}^{\lfloor k/2 \rfloor} h(-4P^{2(\frac{k}{2} - j)}N) + \sum_{j=0}^{\lfloor k/2 \rfloor} h(-P^{2(\frac{k}{2} - j)}N)}{2} - \left(\sum_{i=1}^k P^i\right) \\
& \hspace{1cm} + \sum_{0<r\leq 2\sqrt{\frac{P^k}{N}}} \sum_{d^2\mid r^2N - 4P^{k}} h(N(r^2N - 4P^{k})/d^2) + O(1).
\end{align*}
When $k=2$ the trace formula further simplifies to
\begin{align*}
\sum_{f\in H^{\new}(N)} P \lambda_f(P^2)\epsilon(f) & =  \frac{h(-4P^2N) + h(-P^2N) + h(-4N) + h(-N)} {2} - (P+P^2) \\
& \hspace{1cm} + \sum_{0<r\leq 2P/\sqrt{N}} \sum_{d^2\mid r^2N - 4P^{2}} h(N(r^2N - 4P^{2})/d^2) + O(1).
\end{align*}

\section{Main Term Calculation via Averages of Gauss Class Numbers I}
\label{Section: Main Term Calculation 1}

The goal of this section is to estimate the Hurwitz class number $H_1(-4P^2N)$ which (as we have noted above) can be written as a sum of four Gauss class numbers.

Recall that if $-d\equiv 0 \text{ or }1\pmod{4}$ then the Gauss class number can be written in terms of a special value of an $L$-function.
More precisely,
\[
h(-d) = \frac{\sqrt{d}}{\pi} L(1, \chi_d)
\]
where $\chi_d$ is the quadratic Dirichlet character of modulus $d$ or $4d$ which can be calculated explicitly via the Kronecker symbol and $L(1, \chi_d)$ is the value at 1 of the Dirichlet series for the Kronecker symbol $\left( \frac{d}{\cdot}\right)$.
The calculation in this section will be separated into two cases, first when $-d\equiv 1\pmod{4}$ and second when $-d\equiv 0\pmod{4}$. 
The first calculation will allow us to estimate averages of  $h(-P^2N)$ and $h(-N)$ and the second will allow us to estimate averages of $h(-4P^2N)$ and $h(-4N)$.

When $\chi$ is a non-principal Dirichlet character of modulus $d$, for a truncation parameter $T$ we know 
\[
L(1,\chi) = \sum_{n=1}^T \frac{\chi(n)}{n} + O\left(\frac{\sqrt{d} \log d}{T}\right).
\]

The main result we want to prove in this section is an analogue of \cite[Proposition~3.1]{Zub23}.

\begin{proposition}
\label{our version of Nina P3.1}
Let $P\neq 2$ be a prime and let $[X,X+Y]$ be an interval of length $Y=o(X)$.
Then as $X\rightarrow \infty$,
{\allowdisplaybreaks\begin{align*}
\frac{\zeta(2) \pi}{XY}\sum_{\substack{N\in [X,X+Y] \\ P\nmid N \\ N \text{ sq-free}}} \frac{h(-P^2 N)}{2} &+ \frac{h(- N)}{2} + \frac{h(-4P^2 N)}{2} + \frac{h(- 4N)}{2} \\
& = \frac{AP}{\sqrt{X}}+ O_\epsilon\left(\left(\frac{1}{\sqrt{X}}+\frac{P^{\frac{7}{6}}X^{\frac{1}{12} }}{Y^{\frac{5}{6}}}+\frac{PY}{X^{\frac32}}\right) (PXY)^\epsilon\right).
\end{align*}}
Here the error term $\rightarrow 0$ as $P,X\rightarrow\infty$ and $Y=o\left(\frac{X}{\log P^2 X}\right)$.
\end{proposition}

\begin{proof}
We prove that for a cut-off parameter $T$
\begin{align*}
\frac{\zeta(2) \pi}{XY}\sum_{\substack{N\in [X,X+Y] \\ P\nmid N \\ N \text{ sq-free}}} \frac{h(-P^2 N)}{2} &+ \frac{h(- N)}{2} + \frac{h(-4P^2 N)}{2} + \frac{h(- 4N)}{2} = \eqref{3-1-1 main equation cutoff} + \eqref{3-1-2 main equation cut off} + \eqref{3-2-1 main equation cut off} + \eqref{3-2-2 main equation cutoff}\\
& = \frac{AP}{\sqrt{X}} + O_{\epsilon}\left(\frac{1}{\sqrt{X}} +  \frac{P}{\sqrt{TX}} + \frac{PT^{\frac{1}{5}+\epsilon}X^{\frac{1}{10}+\epsilon}}{Y} + \frac{P Y\log T}{X^{\frac{3}{2}}} + \frac{P^2\log(P^2 X)}{T}\right).
\end{align*}
Choosing $T = \frac{(YP)^{\frac{5}{6} - \epsilon}}{X^{\frac{1}{12}}}$, the right side of the above expression becomes
\[
\frac{AP}{\sqrt{X}} +  O_\epsilon\left(\frac{1}{\sqrt{X}}+\frac{P^{\frac{7}{12}}}{Y^{\frac{5}{12}}X^{\frac{11}{24}}}+ \frac{(PY)^{1+\epsilon}}{X^{\frac32 -\epsilon }} + \frac{P^{\frac{7}{6}+\epsilon}X^{\frac{1}{12} + \epsilon}}{Y^{\frac{5}{6}-\epsilon}}\right).
\]
and the result follows.
\end{proof}

\subsection{When \texorpdfstring{$-d \equiv 1\pmod{4}$}{}}

Throughout this section $P\neq 2$ and $[X, X+Y]$ is an interval of length $Y=o(X)$.
We provide the following estimate for the first two terms in Proposition~\ref{our version of Nina P3.1}.

\begin{equation*}
\frac{\zeta(2)\pi}{YX}\sum_{\substack{N\in [X, X+Y] \\ P\nmid N\\ N \text{ 
sq-free}}} \frac{h(-P^2N)}{2} + \frac{h(-N)}{2} 
= \frac{2A}{11} \frac{P}{\sqrt{X}} + O_\epsilon\left(\frac{1}{\sqrt{X}}+\frac{P^{\frac{7}{12}}}{Y^{\frac{5}{12}}X^{\frac{11}{24}}}+\frac{P^{\frac{7}{6} + \epsilon}}{Y^{\frac{5}{6}- \epsilon}X^{\frac{1}{12}- \epsilon}}+\frac{P^{\frac{7}{6}+\epsilon}X^{\frac{1}{12} + \epsilon}}{Y^{\frac{5}{6}-\epsilon}}\right). 
\end{equation*}

Here, the main term contribution is from the estimates in Lemma~\ref{3-1-1 main equation}.
The leading term of the expression in Lemma~\ref{3-1-2 main equation} contributes $\frac{1}{\sqrt{X}}$ to the error term of this summation and subsumes the error term $\frac{1}{P\sqrt{X}}$ arising in Lemma~\ref{3-1-1 main equation}.

\subsubsection{Averages of \texorpdfstring{$h(-P^2N)$}{}}
\label{sec 3-1-1}

\begin{lemma}
\label{3-1-1 main equation}
With notation as introduced above,
{\allowdisplaybreaks
\begin{equation*}
\frac{\zeta(2)\pi}{YX}\sum_{\substack{N\in [X, X+Y] \\ P\nmid N \\ N \text{ 
sq-free}}} \frac{h(-P^2N)}{2} = \frac{2AP}{11\sqrt{X}} + O_\epsilon\left(\frac{1}{P\sqrt{X}}+\frac{P^{\frac{7}{12}}}{Y^{\frac{5}{12}}X^{\frac{11}{24}}}+\frac{(PY)^{1+\epsilon}}{X^{\frac32 -\epsilon }} + \frac{P^{\frac{7}{6}+\epsilon}X^{\frac{1}{12} + \epsilon}}{Y^{\frac{5}{6}-\epsilon}}\right).
\end{equation*}
}
\end{lemma}

The first step is to obtain an expression for $\frac{\zeta(2)\pi}{YX}\sum h(-P^2N)$ in terms of a cut-off parameter $T$ where the sum runs over square-free $N \in [X, X+Y]$ such that  $P\nmid N$.
To ensure that the error-term is smaller than the main term, we choose the cut-off parameter appropriately.
The proof of the lemma occupies the remainder of this section.

\begin{proof}
We calculate the averages in intervals using the Gauss class number formula, i.e.,
{\allowdisplaybreaks
\begin{align*}
\frac{1}{P\sqrt{X}} \sum_{\substack{N\in [X, X+Y] \\ P\nmid N \\ N \text{ 
sq-free}}} h(-P^2N) &= \frac{1}{\pi} \sum_{\substack{N\in [X, X+Y]\\ P\nmid N \\ N \text{ sq-free} \\ N \equiv 3\pmod{4}}} \sqrt{\frac{N}{X}} L(1, \chi_{-P^2N})\\
&= \frac{1}{\pi}\sum_{\substack{N\in [X, X+Y]\\ P\nmid N \\ N \text{ sq-free} \\ N \equiv 3\pmod{4}}} \sqrt{\frac{N}{X}} \sum_{n=1}^T \frac{\left( \frac{-P^2N}{n}\right)}{n} + O\left(\frac{YP\sqrt{X}\log(P^2 X)}{T}\right)\\
&= \frac{1}{\pi} \sum_{\substack{N\in [X, X+Y]\\ P\nmid N \\ N \text{ sq-free} \\ N \equiv 3\pmod{4}}} \sum_{m=1}^{\sqrt{T}}   \frac{\sqrt{\frac{N}{X}}\left( \frac{-P^2N}{m^2}\right)}{m^2} + \frac{1}{\pi} \sum_{\substack{N\in [X, X+Y]\\ P\nmid N \\ N \text{ sq-free} \\ N \equiv 3\pmod{4}}} \sum_{\substack{n=1 \\ n\neq \square}}^T   \frac{\sqrt{\frac{N}{X}} \left( \frac{-P^2N}{n}\right)}{n} \\ 
\quad & \quad + O\left(\frac{YP\sqrt{X}\log(P^2 X)}{T}\right).
\end{align*}}

\noindent where to show the second equality we use the fact that $\chi_{-P^2N}$ is always a non-principal character when $N$ is square-free and $P\nmid N$.

Next, we calculate the two (double) sums appearing in the above expression separately.
Note that the first sum contains principal characters.

{
\allowdisplaybreaks
\begin{align*}
\Sq &= \frac{1}{\pi}\sum_{\substack{N\in [X, X+Y]\\ P\nmid N \\ N \text{ sq-free} \\ N \equiv 3\pmod{4}}} \sum_{m=1}^{\sqrt{T}}   \frac{\sqrt{\frac{N}{X}}\left( \frac{-P^2N}{m^2}\right)}{m^2} \\
& =\frac{1}{\pi}\sum_{m=1}^{\sqrt{T}} \frac{1}{m^2}\sum_{\substack{N\in [X, X+Y]\\ P\nmid N \\ N \equiv 3\pmod{4}}}    \mu^2(N)\left( \frac{-P^2N}{m^2}\right)\left(1 + \left(\sqrt{1 + \frac{N-X}{X}} - 1\right) \right) \\
&= \frac{1}{\pi}\sum_{m=1}^{\sqrt{T}} \frac{1}{m^2} \sum_{\substack{N\in [X, X+Y]\\ P\nmid N \\ N \equiv 3\pmod{4}}} \mu^2(N)\left( \frac{P^2}{m^2}\right)\left( \frac{N}{m^2}\right) + O\left(Y \left(\sqrt{1 + \frac{Y}{X}} - 1\right)\right)\\
&= \frac{1}{\pi}\sum_{\substack{m=1\\ \gcd(P,m)=1}}^{\sqrt{T}} \frac{1}{m^2} \sum_{\substack{N\in [X, X+Y]\\ P\nmid N \\ N \equiv 3\pmod{4}}} \mu^2(N)\left( \frac{N}{m^2}\right) + O\left(Y \left(\sqrt{1 + \frac{Y}{X}} - 1\right)\right)\\
&= \frac{1}{\pi} \sum_{\substack{m=1\\ \gcd(P,m)=1}}^{\sqrt{T}} \frac{1}{m^2} \left(\sum_{N\in [X, X+Y]}    \mu^2(N)\left( \frac{N}{m^2}\right)\frac{\chi_1(N) - \chi_2(N)}{2}\right) + O\left(\frac{Y^2}{X}\right)\\
&= \frac{1}{\pi}\sum_{\substack{m=1\\ \gcd(P,m)=1}}^{\sqrt{T}} \frac{Y}{\zeta(2)} \frac{\eta(2m)}{2m^2} + O_{\epsilon}\left( m^{\frac{1}{5}+\epsilon}X^{\frac{3}{5}+\epsilon}\frac{1}{m^2}\right) + O\left(\frac{Y^2}{X}\right) \\
&= \frac{4YA}{11\pi\zeta(2)} + O_{\epsilon}\left( \frac{Y}{P^2} + \frac{Y}{\sqrt{T}} + X^{\frac{3}{5}+\epsilon} +\frac{Y^2}{X}\right) \text{ by Lemma~\ref{eta} where } A:= \prod_{p}\left(1 + \frac{p}{(p+1)^2 (p-1)} \right).
\end{align*}}
Here, $\chi_1$ and $\chi_2$ are characters modulo 4, and $\chi_1$ is principal.
The character $\left(\frac{N}{m^2}\right)\chi_1(N)$ is principal modulo $2m$ whereas, $\left(\frac{N}{m^2}\right)\chi_2(N)$ is always non-principal modulo $4m$.
Note that for the second last equality we use \cite[Lemma~6.7]{Zub23}.
Next we calculate the non-square term in a manner identical to the one above.

{
\allowdisplaybreaks
\begin{align*}
    \NSq &= \frac{1}{\pi} \sum_{\substack{N\in [X, X+Y]\\ P\nmid N \\ N \text{ sq-free} \\ N \equiv 3\pmod{4}}} \sum_{\substack{n=1 \\ n\neq \square}}^T \frac{\sqrt{\frac{N}{X}} \left( \frac{-P^2N}{n}\right)}{n}\\
    & = \frac{1}{\pi} \sum_{\substack{N\in [X, X+Y]\\ P\nmid N \\ N \text{ sq-free} \\ N \equiv 3\pmod{4}}} \sum_{\substack{n=1 \\ n\neq \square}}^{T} \frac{1}{n}  \left( \frac{-P^2N}{n}\right)\left(1 + \left(\sqrt{1 + \frac{N-X}{X}} - 1\right) \right) \\
    &=\frac{1}{\pi} \sum_{\substack{N\in [X, X+Y]\\ P\nmid N \\ N \text{ sq-free} \\ N \equiv 3\pmod{4}}} \sum_{\substack{n=1 \\ n\neq \square}}^{T} \frac{\left( \frac{-P^2N}{n}\right)}{n} + O\left(\frac{Y^2\sum_{n\leq T} \frac{1}{n}}{X} \right) \\
    &=\frac{1}{\pi} \sum_{\substack{N\in [X, X+Y]\\ P\nmid N \\ N \text{ sq-free} \\ N \equiv 3\pmod{4}}}  \sum_{\substack{n=1 \\ n\neq \square \\ \gcd(P,n)=1}}^{T} \frac{\left( \frac{-N}{n}\right)}{n} + O\left(\frac{Y^2\log T}{X} \right)  \\
    &\leq \frac{1}{\pi} \sum_{\substack{n=1 \\ n\neq \square}}^{T} \sum_{\substack{N\in [X, X+Y]\\ N \equiv 3\pmod{4} \\ N \text{ sq-free}}}   \frac{\left( \frac{-N}{n}\right)}{n} + O\left(\frac{Y^2\log T}{X} \right)  \\
    &=\frac{1}{\pi} \sum_{\substack{n=1 \\ n\neq \square}}^{T}\frac{\left( \frac{-1}{n}\right)}{n} \left( \sum_{\substack{N\in [X, X+Y] \\ N \text{ sq-free}}}   \left(\frac{N}{n}\right)   \frac{\chi_{1}(N)-\chi_2(N)}{2}\right) + O\left(\frac{Y^2\log T}{X} \right) \\
    &=\frac{1}{\pi} \sum_{\substack{n=1 \\ n\neq \square}}^{T} \frac{1}{n} O_{\epsilon}\left(n^{\frac{1}{5}+\epsilon}X^{\frac{3}{5}+\epsilon}\right) + O\left(\frac{Y^2\log T}{X} \right) \\
    &\ll_{\epsilon} T^{\frac{1}{5}+\epsilon}X^{\frac{3}{5}+\epsilon} + \frac{Y^2\log T}{X}.
\end{align*}}

Since $N$ is not a square, note that $\left(\frac{N}{n}\right)$ is non-principal.
Since $\left(\frac{N}{2}\right)$ is primitive modulo 8, the characters $\left(\frac{N}{n}\right)\chi_1(N)$ and $\left(\frac{N}{n}\right)\chi_2(N)$ are non-principal.
For the last equality, we use \cite[Lemma~6.7]{Zub23}.

In conclusion we obtain that for a truncation parameter $T$,

{\allowdisplaybreaks
\begin{equation}
\label{3-1-1 main equation cutoff}
\begin{split}
\frac{\zeta(2)\pi}{YX}\sum_{\substack{N\in [X, X+Y] \\ P\nmid N \\ N \text{ 
sq-free}}} \frac{h(-P^2N)}{2} & = \frac{2AP}{11\sqrt{X}} \\ 
& + O_\epsilon\left( \frac{1}{P \sqrt{X}} + \frac{P}{\sqrt{TX}} + \frac{PT^{\frac{1}{5}+\epsilon}X^{\frac{1}{10}+\epsilon}}{Y} + \frac{P Y\log T}{X^\frac{3}{2}} + \frac{P^2\log(P^2 X)}{T}\right).
\end{split}
\end{equation}
}

Choosing $T = \frac{(PY)^{\frac{5}{6}}}{X^{\frac{1}{12}}}$, completes the proof of the lemma.
\end{proof}

\subsubsection{Averages of \texorpdfstring{$h(-N)$}{}}

\begin{lemma}
\label{3-1-2 main equation} 
With notation as before
\begin{equation*}
\frac{\zeta(2)\pi}{YX}\sum_{\substack{N\in [X, X+Y] \\ P\nmid N \\ N \text{ 
sq-free}}} \frac{h(-N)}{2}  = \frac{2A}{11\sqrt{X}} + O_\epsilon\left(\frac{1}{(PY)^{\frac{5}{12}}X^{\frac{11}{24}}} + \frac{P^{\frac{1}{6}  + \epsilon}X^{\frac{1}{12} + \epsilon}}{Y^{\frac{5}{6} - \epsilon}
}+\frac{Y^{1+\epsilon}P^\epsilon}{X^{3/2-\epsilon}}+\frac{X^{\frac{1}{12}+\epsilon}}{Y^{\frac{5}{6}-\epsilon}P^{\frac{5}{6}-\epsilon}}
\right).
\end{equation*}
\end{lemma}

\begin{proof}
We begin by calculating the average
{\allowdisplaybreaks
\begin{align*}
\frac{1}{\sqrt{X}} \sum_{\substack{N\in [X, X+Y] \\ P\nmid N \\ N \text{ sq-free}}} h(-N) &= \frac{1}{\pi}\sum_{\substack{N\in [X, X+Y]\\ P\nmid N \\ N \equiv 3\pmod{4} \\ N \text{ sq-free}}}\sqrt{\frac{N}{X}} L(1, \chi_{-N})\\
& = \frac{1}{\pi}\sum_{\substack{N\in [X, X+Y]\\ P\nmid N \\ N \equiv 3\pmod{4} \\ N \text{ sq-free}}}\sqrt{\frac{N}{X}} \sum_{n=1}^T \frac{\left( \frac{-N}{n}\right)}{n} + O\left(\frac{Y\sqrt{X}\log(X)}{T}\right) \\
&= \frac{1}{\pi}\sum_{\substack{N\in [X, X+Y]\\ P\nmid N \\ N \equiv 3\pmod{4} \\ N \text{ sq-free}}} \sum_{m=1}^{\sqrt{T}}   \frac{\sqrt{\frac{N}{X}}\left( \frac{-N}{m^2}\right)}{m^2} + \frac{1}{\pi}\sum_{\substack{N\in [X, X+Y]\\ P\nmid N \\ N \equiv 3\pmod{4} \\ N \text{ sq-free}}}\sum_{\substack{n=1 \\ n\neq \square}}^T   \frac{\sqrt{\frac{N}{X}} \left( \frac{-N}{n}\right)}{n} \\
& \quad + O\left(\frac{Y\sqrt{X}\log X}{T}\right).
\end{align*}}
where to show the second equality we use the fact that $\chi_{-N}$ is always a non-principal character when $N$ is square-free.
We proceed as before to obtain the estimates.
First we work with the square terms
{\allowdisplaybreaks
\begin{align*}
\Sq &= \frac{1}{\pi}\sum_{\substack{N\in [X, X+Y]\\ P\nmid N \\ N \equiv 3\pmod{4} \\ N \text{ sq-free}}} \sum_{m=1}^{\sqrt{T}}   \frac{\sqrt{\frac{N}{X}}\left( \frac{-N}{m^2}\right)}{m^2}\\
&=\frac{1}{\pi}\sum_{m=1}^{\sqrt{T}} \frac{1}{m^2}\sum_{\substack{N\in [X, X+Y]\\ P\nmid N \\ N \equiv 3\pmod{4} \\ N \text{ sq-free}}}   \left( \frac{-N}{m^2}\right)\left(1 + \left(\sqrt{1 + \frac{N-X}{X}} - 1\right) \right) \\
&= \frac{1}{\pi}\sum_{m=1}^{\sqrt{T}} \frac{1}{m^2}\sum_{\substack{N\in [X, X+Y]\\ P\nmid N \\ N \equiv 3\pmod{4} \\ N \text{ sq-free}}} \left( \frac{-N}{m^2}\right) + O\left(Y \left(\sqrt{1 + \frac{Y}{X}} - 1\right)\right)\\
&= \frac{1}{\pi}\sum_{m=1}^{\sqrt{T}} \frac{1}{m^2}\sum_{\substack{N\in [X, X+Y]\\ P\nmid N \\ N \equiv 3\pmod{4} \\ N \text{ sq-free}}} \left( \frac{N}{m^2}\right) + O\left(Y \left(\sqrt{1 + \frac{Y}{X}} - 1\right)\right)\\
&= \frac{1}{\pi}\sum_{m=1}^{\sqrt{T}} \frac{1}{m^2} \left(\sum_{\substack{N\in [X, X+Y] \\ P\nmid N}}    \mu^2(N)\left( \frac{N}{m^2}\right)\frac{\chi_1(N) - \chi_2(N)}{2}\right) + O\left(\frac{Y^2}{X}\right)\\
&= \frac{1}{\pi}\sum_{m=1}^{\sqrt{T}} \frac{Y}{\zeta(2)} \frac{\eta(2m)}{2m^2} + O_{\epsilon}\left( \frac{m^{\frac{1}{5}+\epsilon}X^{\frac{3}{5}+\epsilon}}{m^2}\right) + O\left(\frac{Y^2}{X}\right)\\
&= \frac{1}{\pi}\frac{4YA}{11\zeta(2)} + O_{\epsilon}\left(\frac{Y}{\sqrt{T}} + X^{\frac{3}{5}+\epsilon} + \frac{Y^2}{X}\right) \text{by Lemma~\ref{eta}}. 
\end{align*}
}Here, $\chi_1$ and $\chi_2$ are characters modulo 4, and $\chi_1$ is principal.
The character $\left(\frac{N}{m^2}\right)\chi_1(N)$ is principal modulo $2m$ whereas, $\left(\frac{N}{m^2}\right)\chi_2(N)$ is non-principal modulo $4m$. 
For the second last equality we use \cite[Lemma~6.7]{Zub23}.
Next we calculate the non-square terms.

{\allowdisplaybreaks\begin{align*}
    \NSq &= \frac{1}{\pi}\sum_{\substack{N\in [X, X+Y]\\ P\nmid N \\ N \equiv 3\pmod{4}}}\sum_{\substack{n=1 \\ n\neq \square}}^T   \frac{\sqrt{\frac{N}{X}} \left( \frac{-N}{n}\right)}{n}\\
    &= \frac{1}{\pi}  \sum_{\substack{N\in [X, X+Y]\\ P\nmid N \\ N \equiv 3\pmod{4}}}  \sum_{\substack{n=1 \\ n\neq \square}}^{T} \frac{1}{n}  \left( \frac{-N}{n}\right)\left(1 + \left(\sqrt{1 + \frac{N-X}{X}} - 1\right) \right) \\
    &=\frac{1}{\pi} \sum_{\substack{N\in [X, X+Y]\\ P\nmid N \\ N \equiv 3\pmod{4}}}     \sum_{\substack{n=1 \\ n\neq \square}}^{T} \frac{\left( \frac{-N}{n}\right)}{n} + O\left(\frac{Y^2\sum_{n\leq T} \frac{1}{n}}{X} \right) \\
    &=\frac{1}{\pi} \sum_{\substack{n=1 \\ n\neq \square}}^{T}\frac{\left( \frac{-1}{n}\right)}{n} \left( \sum_{\substack{N\in [X, X+Y] \\ P\nmid N}}   \left(\frac{N}{n}\right)   \frac{\chi_{1}(N)-\chi_2(N)}{2}\right) + O\left(\frac{Y^2\log T}{X} \right) \\
    &=\frac{1}{\pi} \sum_{\substack{n=1 \\ n\neq \square}}^{T} \frac{1}{n} O_{\epsilon}\left(n^{\frac{1}{5}+\epsilon}X^{\frac{3}{5}+\epsilon}\right) + O\left(\frac{Y^2\log T}{X} \right) \ll_{\epsilon} T^{\frac{1}{5}+\epsilon}X^{\frac{3}{5}+\epsilon} + \frac{Y^2\log T}{X}.
\end{align*}}

\noindent We remind the reader that since $n$ is not a square, we have that $\left(\frac{N}{n}\right)$ is non-principal.
Also, we know that $\left(\frac{N}{2}\right)$ is primitive modulo 8.
Therefore, the characters $\left(\frac{N}{n}\right)\chi_1(N)$ and $\left(\frac{N}{n}\right)\chi_2(N)$ are both non-principal.
We apply \cite[Lemma~6.7]{Zub23} to obtain the last equality.

In conclusion we obtain that for a truncation parameter $T$, we have
\begin{equation}
\label{3-1-2 main equation cut off}
\frac{\zeta(2)\pi}{YX}\sum_{\substack{N\in [X, X+Y] \\ P\nmid N \\ N \text{ 
sq-free}}} \frac{h(-N)}{2}  =\frac{2A}{11\sqrt{X}} + O_{\epsilon}\left( \frac{1}{\sqrt{TX}} + \frac{T^{\frac{1}{5}+\epsilon}X^{\frac{1}{10}+\epsilon}}{Y} + \frac{Y\log T}{X^{\frac{3}{2}}} + \frac{\log(X)}{T} \right). 
\end{equation}
As in \S\ref{sec 3-1-1}, we choose $T = \frac{(PY)^{\frac{5}{6}}}{X^{\frac{1}{12}}}$ to obtain the expression in the statement of the lemma.
\end{proof}

\subsection{When \texorpdfstring{$d \equiv 0\pmod{4}$}{}}
As before, we assume that $P\neq 2$ and that $[X, X+Y]$ is an interval of length $Y=o(X)$.
The calculations in Lemmas~\ref{3-2-1 main equation} and \ref{3-2-2 main equation} will account for the last two terms in Proposition~\ref{our version of Nina P3.1}.
More precisely,

{\allowdisplaybreaks
\begin{equation*}
\frac{\zeta(2)\pi}{YX}\sum_{\substack{N\in [X, X+Y] \\ P\nmid N\\ N \text{ 
sq-free}}} \frac{h(-4P^2N)}{2}  + \frac{h(-4N)}{2} = 
\frac{9A}{11} \frac{P}{\sqrt{X}} + O_\epsilon\left(\frac{1}{\sqrt{X}}+\frac{P^{\frac{7}{12}}}{Y^{\frac{5}{12}}X^{\frac{11}{24}}}+ \frac{(PY)^{1+\epsilon}}{X^{\frac32 -\epsilon }} + \frac{P^{\frac{7}{6}+\epsilon}X^{\frac{1}{12} + \epsilon}}{Y^{\frac{5}{6}-\epsilon}}\right).
\end{equation*}
}

\subsubsection{Averages of $h(-4P^2N)$}

\begin{lemma}
\label{3-2-1 main equation} 
With notation as above
{\allowdisplaybreaks
\begin{equation*}
\frac{\zeta(2)\pi}{YX} \sum_{\substack{N\in [X, X+Y] \\ P\nmid N \\ N \text{ 
sq-free}}} \frac{h(-4P^2N)}{2}  = \frac{9AP}{11\sqrt{X}} + O_\epsilon\left(\frac{1}{P\sqrt{X}}+\frac{P^{\frac{7}{12}}}{Y^{\frac{5}{12}}X^{\frac{11}{24}}}+ \frac{(PY)^{1+\epsilon}}{X^{\frac32 -\epsilon }} + \frac{P^{\frac{7}{6}+\epsilon}X^{\frac{1}{12} + \epsilon}}{Y^{\frac{5}{6}-\epsilon}}\right).
\end{equation*}}
\end{lemma}

\begin{proof}
As we have done previously, consider

{\allowdisplaybreaks
\begin{align*}
\frac{1}{P\sqrt{X}} \sum_{\substack{N\in [X, X+Y] \\ P\nmid N \\ N \text{ 
sq-free}}} h(-4P^2N) &= \frac{2}{\pi} \sum_{\substack{N\in [X, X+Y] \\ P\nmid N \\ N \text{ 
sq-free}}} \sqrt{\frac{N}{X}} L(1, \chi_{-4P^2N})\\
&= \frac{2}{\pi} \sum_{\substack{N\in [X, X+Y] \\ P\nmid N \\ N \text{ 
sq-free}}} \sqrt{\frac{N}{X}} \sum_{n=1}^T \frac{\left( \frac{-4P^2N}{n}\right)}{n} + O\left(\frac{YP\sqrt{X}\log(P^2 X)}{T}\right)\\
&=\frac{2}{\pi} \sum_{\substack{N\in [X, X+Y] \\ P\nmid N \\ N \text{ 
sq-free}}} \sqrt{\frac{N}{X}} \sum_{m=1}^{\sqrt{T}} \frac{\left( \frac{-4P^2N}{m^2}\right)}{m^2} +  \frac{2}{\pi}  \sum_{\substack{N\in [X, X+Y] \\ P\nmid N \\ N \text{ 
sq-free}}}\sqrt{\frac{N}{X}} \sum_{\substack {n=1 \\ n\neq \square}}^T \frac{\left( \frac{-4P^2N}{n}\right)}{n}\\
&+O\left(\frac{YP\sqrt{X}\log(P^2 X)}{T}\right).
\end{align*}
}
As before, we estimate the two double sums separately.
First we consider the square terms

{\allowdisplaybreaks
\begin{align*}
\Sq 
    &=\frac{2}{\pi} \sum_{\substack{N\in [X, X+Y] \\ P\nmid N \\ N \text{ sq-free}}} \sqrt{\frac{N}{X}} \sum_{m=1}^{\sqrt{T}} \frac{\left( \frac{-4P^2N}{m^2}\right)}{m^2}\\
    &=\frac{2}{\pi} \sum_{m=1}^{\sqrt{T}}\frac{1}{m^2} \sum_{\substack{N\in [X, X+Y] \\ P\nmid N \\ N \text{ sq-free}}} \left( \frac{-4P^2N}{m^2}\right)\mu^2(N)\left(1+\sqrt{1+\frac{N-X}{X}}-1\right)\\
    &=\frac{2}{\pi} \sum_{m=1}^{\sqrt{T}}\frac{1}{m^2} \sum_{\substack{N\in [X, X+Y] \\ P\nmid N \\ N \text{ sq-free}}} \left( \frac{-4P^2N}{m^2}\right)\mu^2(N)+O\left(\frac{Y^2}{X}\frac{1}{m^2}\right)\\
    &=\frac{2}{\pi}\sum_{\substack{m=1 \\ \gcd(m,2P)=1}}^{\sqrt{T}}\frac{1}{m^2}\frac{Y}{\zeta(2)}\eta(m)+O_{\epsilon}\left(X^{\frac{3}{5}+\epsilon}m^{\epsilon-2}+\frac{Y^2}{X}\frac{1}{m^2}\right)\\
    &=\frac{1}{\pi} \frac{18YA}{\zeta(2)11}+O_{\epsilon}\left( \frac{Y}{P^2}+\frac{Y}{\sqrt{T}} + X^{\frac{3}{5}+\epsilon}+\frac{Y^2}{X}\right) \text{ by Lemma~\ref{eta}},
\end{align*}
}

\noindent The second last equality follows from \cite[Lemma~6.7]{Zub23}.
Next we estimate the other double sum  to obtain

{\allowdisplaybreaks
\begin{align*}
\NSq 
    &= \frac{2}{\pi}  \sum_{\substack{N\in [X, X+Y] \\ P\nmid N \\ N \text{ sq-free}}} \sqrt{\frac{N}{X}} \sum_{\substack {n=1 \\ n\neq \square}}^T \frac{\left( \frac{-4P^2N}{n}\right)}{n}\\
    &=\frac{2}{\pi}\sum_{\substack {n=1 \\ n\neq \square}}^T\frac{1}{n}\sum_{\substack{N\in [X, X+Y] \\ P\nmid N }}\left( \frac{-4P^2N}{n}\right)\mu^2(N)\left(1+\sqrt{1+\frac{N-X}{X}}-1\right)\\
    &=\frac{2}{\pi}\sum_{\substack {n=1 \\ n\neq \square}}^T\frac{1}{n}\sum_{\substack{N\in [X, X+Y] \\ P\nmid N }}\left( \frac{-4P^2N}{n}\right)\mu^2(N)+O\left(\frac{Y^2\log(T)}{X}\right)\\
    &=\frac{2}{\pi}\sum_{\substack {n=1 \\ n\neq \square}}^T\frac{1}{n}\left(\frac{-4P^2}{n}\right)\sum_{\substack{N\in [X, X+Y] \\ P\nmid N}}\left( \frac{N}{n}\right)\mu^2(N)+O\left(\frac{Y^2\log(T)}{X}\right)\\
    &=\frac{2}{\pi}\sum_{\substack {n=1 \\ n\neq \square}}^T\frac{1}{n}\left(\frac{-4P^2}{n}\right) O_{\epsilon}\left(X^{\frac{3}{5}+\epsilon}n^{\frac{1}{5}+\epsilon}\right)+O\left(\frac{Y^2\log(T)}{X}\right)\\
    & \ll_{\epsilon} T^{\frac{1}{5}+\epsilon}X^{\frac{3}{5} + \epsilon}+\frac{Y^2\log(T)}{X}.
\end{align*}
}
For a truncation parameter $T$, we have
{\allowdisplaybreaks
\begin{equation}
\label{3-2-1 main equation cut off}
\begin{split}
\frac{\zeta(2)\pi}{YX} \sum_{\substack{N\in [X, X+Y] \\ P\nmid N \\ N \text{ 
sq-free}}} \frac{h(-4P^2N)}{2}  & = \frac{9AP}{11\sqrt{X}} \\
& + O_\epsilon\left( \frac{1}{P\sqrt{X}} + \frac{P}{\sqrt{TX}} + \frac{PT^{\frac{1}{5}+\epsilon}X^{\frac{1}{10}+\epsilon}}{Y} + \frac{P Y\log T}{X^{\frac{3}{2}}} + \frac{P^2\log(P^2 X)}{T}\right).
\end{split}
\end{equation}
}
The lemma follows by choosing $T=\frac{(PY)^{\frac{5}{6}}}{X^{\frac{1}{12}}}$.
\end{proof}

\subsubsection{Averages of $h(-4N)$}

\begin{lemma}
\label{3-2-2 main equation}
With notation as before,
\[\frac{\zeta(2)\pi}{YX} \sum_{\substack{N\in [X, X+Y] \\ P\nmid N \\ N \text{ sq-free}}} \frac{h(-4N)}{2} = \frac{9A}{22\sqrt{X}} + O_\epsilon\left(\frac{1}{(PY)^{\frac{5}{12}}X^{\frac{11}{24}}} + \frac{P^{\frac{1}{6}  + \epsilon}X^{\frac{1}{12} + \epsilon}}{Y^{\frac{5}{6} - \epsilon}
}+\frac{Y^{1+\epsilon}P^\epsilon}{X^{\frac{3}{2}-\epsilon}}+\frac{X^{\frac{1}{12}+\epsilon}}{Y^{\frac{5}{6}-\epsilon}P^{\frac{5}{6}-\epsilon}}
\right).
\]
\end{lemma}

\begin{proof}
The idea of the proof is exactly the same as before.    
{\allowdisplaybreaks
\begin{align*}
    \frac{1}{\sqrt{X}} \sum_{\substack{N\in [X, X+Y] \\ P\nmid N \\ N \text{sq-free}}} h(-4N) &= \frac{2}{\pi}\sum_{\substack{N\in [X, X+Y] \\ P\nmid N}}\sqrt{\frac{N}{X}} L(1, \chi_{-4N})\\
    &= \frac{2}{\pi} \sum_{\substack{N\in [X, X+Y] \\ P\nmid N \\ N \text{sq-free}}} \sqrt{\frac{N}{X}} \sum_{n=1}^T \frac{\left( \frac{-4N}{n}\right)}{n} + O\left(\frac{Y\sqrt{X}\log( X)}{T}\right)\\
    &=\frac{2}{\pi} \sum_{\substack{N\in [X, X+Y] \\ P\nmid N \\ N \text{sq-free}}} \sqrt{\frac{N}{X}} \sum_{m=1}^{\sqrt{T}} \frac{\left( \frac{-4N}{m^2}\right)}{m^2} +  \frac{2}{\pi} \sum_{\substack{N\in [X, X+Y] \\ P\nmid N \\ N \text{sq-free}}} \sqrt{\frac{N}{X}} \sum_{\substack {n=1 \\ n\neq \square}}^T \frac{\left( \frac{-4N}{n}\right)}{n}\\
    & + O\left(\frac{Y\sqrt{X}\log X}{T}\right).
\end{align*}
}

\noindent For the first double sum which are the `square terms', we obtain the following expression where we note that the sum is over all $m$ that is odd

{
\allowdisplaybreaks
\begin{align*}
    \Sq &= \frac{2}{\pi} \sum_{\substack{N\in [X, X+Y] \\ P\nmid N \\ N \text{sq-free}}} \sqrt{\frac{N}{X}} \sum_{m=1}^{\sqrt{T}} \frac{\left( \frac{-4N}{m^2}\right)}{m^2} \\
    &=\frac{2}{\pi} \sum_{\substack{m=1\\\gcd(m,2)=1}}^{\sqrt{T}}\frac{1}{m^2}\frac{Y}{\zeta(2)}\eta(m)+O_{\epsilon}\left(\frac{X^{\frac{3}{5} + \epsilon}}{m^{2-\epsilon}} + \frac{Y^2}{X}\frac{1}{m^2}\right)\\
    & = \frac{9YA}{11\pi \zeta(2)}+O_{\epsilon}\left( \frac{Y}{\sqrt{T}} + X^{\frac{3}{5}+\epsilon} + \frac{Y^2}{X}\right) \text{ by Lemma~\ref{eta}}.
\end{align*}
}
The second double sum which are the `non-square' terms can be estimated as before and we obtain
\begin{align*}
\NSq & = \frac{2}{ \pi} \sum_{\substack{N\in [X, X+Y] \\ P\nmid N \\ N \text{ sq-free}}} \sqrt{\frac{N}{X}} \sum_{\substack {n=1 \\ n\neq \square}}^T \frac{\left( \frac{-4N}{n}\right)}{n} 
\ll_{\epsilon}\left( T^{\frac{1}{5}+\epsilon}X^{\frac{3}{5}+\epsilon} + \frac{Y^2\log(T)}{X}\right).
\end{align*}

Therefore,
\begin{equation}
\label{3-2-2 main equation cutoff}
\frac{\zeta(2)\pi}{YX} \sum_{\substack{N\in [X, X+Y] \\ P\nmid N \\ N \text{ sq-free}}} \frac{h(-4N)}{2}  = \frac{9A}{22\sqrt{X}} + O_{\epsilon}\left( \frac{1}{\sqrt{TX}} + \frac{T^{\frac{1}{5}+\epsilon}X^{\frac{1}{10}+\epsilon}}{Y} + \frac{Y\log T}{X^{\frac{3}{2}}} + \frac{\log(X)}{T} \right). 
\end{equation}
The result follows from choosing $T = \frac{(PY)^\frac{5}{6}}{X^{\frac{1}{12}}}$. 
\end{proof}

\section{Remainder analysis} 
\label{Section: Remainder}

In this section, we work under the assumption that $P\neq 2$ is a prime, $r$ is a positive integer, and $X>Y>0$ satisfies $r^2(X+Y)<4P^2$.
Given positive integers $r$ and $d$, define the set 
\[
\mathcal{A}_{r,d}=\left\{N\in \mathbb{Z}  \ \Big\vert \ N \textup{ square-free}, \ \gcd(P,N)=1,  \ d^2\mid r^2N-4P^2, \ \frac{r^2N^2-4P^2N}{d^2}\equiv 0 \textup{ 
or }1\pmod 4\right\}.
\]
First, we calculate the size of the above set.
We do so by breaking the calculation into several cases:
\begin{itemize}
    \item[\textbf{($r$ odd)}]
    \label{r odd} 
    Assume that $r$ is odd.
    Then $r^2N-4P^2$ is not divisible by $4$ as $N$ is square-free.
    Thus, any divisor $d$ satisfying $d^2\mid (r^2N-4P^2)$ has to be odd as well.

    Observe that,
    \[
    r^2 N^2 - 4P^2 N\equiv 0 \text{ or } 1\pmod 4.
    \]
    As $d$ is odd, it suffices to determine whether we can solve the congruence
    \[
    r^2N\equiv 4P^2\pmod{d^2}.
    \]
    There is a unique solution when $\gcd(d,r)=\gcd(d,2)=1$ and no solutions otherwise.
    
    \item[\textbf{($r$ even)}]
    \label{r even} 
    Assume that $r$ is even.
    
    \noindent \emph{Claim:} $\gcd(P,d)=\gcd(P,r)=1$. \newline
    \emph{Justification:} By definition $r^2N\le 4P^2$.
    As $N\ge X>4$, we deduce that $r<P$.
    If $P\mid d$, we obtain that $P^2\mid r^2N$.
    As $N$ is square-free, this implies $P\mid r$ which is impossible.
    This proves the claim. 

    Let us write $r=2\ell$.
    The condition 
    \[
    \frac{r^2N^2-4P^2N}{d^2}\equiv 0\text{ or }1\pmod{4}
    \]
    is equivalent to the existence of an integer $k$ such that 
    $4\ell^2 N - 4P^2=kd^2$ and $kN\equiv 0 \text{ or } 1\pmod {4}$.
    This calculation can be further divided into two cases.
    \begin{itemize}
        \item[($d$ odd)] If $d$ is odd, then $4\mid k$.
        Let $k=4k'$.
        Thus, we look for an $N$ that satisfies
        \[
        \ell^2 N - P^2=k'd^2
        \]
        which is equivalent to $\ell^2N\equiv P^2 \pmod{d^2}$.
        This congruence has exactly one solution if $\gcd(\ell,d)=1$ and zero solutions otherwise. 
        \item[($d$ even)] Assume that $d=2b$ is even.
        As $N$ is square-free it is either odd or congruent to $2$ modulo $4$.
        
        If $N$ is even, then $4\ell^2N$ and $4kb^2$ have to be divisible by $8$; but $8 \nmid 4P^2$.
        
        So $N$ must be odd if $d$ is even.
        As $kN\equiv 0 \text{ or }1\pmod{4}$ we will distinguish two more cases
        \begin{itemize}
            \item[($4 \mid k$)] In this situation we try to solve 
            \[
            \ell^2 N - P^2\equiv 0\mod {4b^2}.
            \]
            This congruence has a solution if and only if $\ell$ is coprime to $2b$ and no solutions otherwise.
            \item[($4 \nmid k$)]
            We now consider the case that $N\equiv k^{-1}\equiv k\pmod 4$.
            Writing $k=N+4k'$ we obtain
            \[
            \ell^2N-P^2=Nb^2 + 4k'b^2. 
            \]
            This is equivalent to $N(\ell^2-b^2)\equiv P^2 \pmod{4b^2}$. This congruence has a solution if $\ell$ and $b$ are coprime and exactly one of them is even.
            In all other cases there are no solutions. 
        \end{itemize}
    \end{itemize}
\end{itemize}

In summary we obtain that $\mathcal{A}_{r,d}=\left\{N\in \mathbb{Z}  \ \Big\vert \ N \textup{ square-free}, \ \gcd(P,N)=1, \ N\pmod{d^2}\in \mathcal{R}_{r,d}\right\}$, where 
\[
\abs{\mathcal{R}_{r,d}} = 
\begin{cases} 
1 & \text{ if } (r,d)=(d,2)=1 \\
1 & \text{ if }r \text{ even}, \  (r,d)=1 \\
1 & \text{ if } (r,d)=2, \ 4\nmid d\\
2 & \text{ if } r \text{ even}, \ (r,d)=2, \ 4\mid d \\
0 & \text{ otherwise.}
\end{cases}
\]
The two elements in the fourth point appear as we get one solution in case ($r$ even)--($d$ even)--($4 \mid k$) and one solution for the case ($r$ even)--($d$ even)--($4 \nmid k$).

\section{Main Term Calculation via Averages of Gauss Class Numbers II}
\label{Section: Main Term Calculation 2}

Throughout this section we assume that $P^2\ll X^{1+\delta_2}$.

For $N$ a square-free positive integer and $P$ a prime satisfying $\gcd(P,N)=1$, the goal of this section is to calculate $\sum H_1(r^2N^2 - 4P^2N)$ with the sum of running over $0<r\leq 2P/\sqrt{X+Y}$.
As observed before, the summation can be expressed as a double sum of Gauss class numbers which in turn can be written explicitly in terms of special values of $L$-functions.
More precisely, for a divisor $d^2\mid (r^2 N - 4P^2)$ satisfying $\frac{r^2 N^2 - 4P^2 N}{d^2} \equiv 0 \text{ or }1\pmod{4}$, the class number formula asserts that
\[
h\left( \frac{r^2 N^2 - 4P^2 N}{d^2} \right) = \frac{\sqrt{4P^2 N - r^2 N^2}}{\pi d} L(1, \chi_{\frac{r^2 N^2 - 4P^2 N}{d^2}}).
\]
Therefore, as $r$ varies in the range $1$ to $ 2P/\sqrt{X+Y}$, we have 
\[
\sum_{\substack{N\in [X, X+Y] \\ P\nmid N \\ N \text{ sq-free}}} H_1(r^2 N^2 - 4P^2 N) = \frac{1}{\pi}\sum_{d< 2P} \sum_{\substack{N\in [X,X+Y]\\ N\in \mathcal{A}_{r,d}}} \frac{L\left(1, \chi_{\frac{r^2 N^2 - 4P^2 N}{d^2}}\right)}{d}  \sqrt{4P^2 N - r^2 N^2}.
\]
Note that 

{\allowdisplaybreaks
\begin{align*}
\sqrt{4P^2N-r^2N^2}-\sqrt{4P^2X-r^2X^2}&=\sqrt{4P^2-r^2N}\left(\sqrt{N}-\sqrt{X}\right)-\sqrt{4P^2X-r^2X^2}\left(1-\sqrt{1-\frac{r^2(N-X)}{4P^2-r^2X}}\right)\\
&\ll \sqrt{4P^2-r^2N}\sqrt{X}\left( \sqrt{1+\frac{Y}{X}}-1\right)-\sqrt{X}\sqrt{P^2-r^2X}\left(1-\sqrt{1-\frac{r^2Y}{4P^2-r^2X}}\right) \\
&\ll P\frac{Y}{\sqrt{X}}+r\sqrt{Y}\sqrt{X}.
\end{align*}
}
For a non-trivial character $\chi$ of conductor $q$, we know by Siegel's bound that $\abs{L(1, \chi)} \ll \log(q)$ (see for example \cite[p.~2]{friedlander}).
Therefore,
{\allowdisplaybreaks
\begin{equation}
\label{to-be.truncated}
\begin{split}
   \sum_{\substack{N\in [X, X+Y] \\ P\nmid N \\ N \text{ sq-free}}} H_1(r^2 N^2 - 4P^2 N) &= \frac{1}{\pi}\sum_{d< 2P} \sum_{\substack{N\in [X,X+Y]\\ N\in \mathcal{A}_{r,d}}} \frac{L\left(1, \chi_{\frac{r^2 N^2 - 4P^2 N}{d^2}}\right)}{d}  \sqrt{4P^2 X - r^2 X^2}\\
    & + O\left(Y^2P^{1+\epsilon}X^{-\frac{1}{2}+\epsilon} + rX^{\frac{1}{2}+\epsilon}Y^{\frac{3}{2}}P^{\epsilon}\right).
\end{split}   
\end{equation}}


We now truncate the main term and obtain 

{\allowdisplaybreaks
\begin{align*}
\eqref{to-be.truncated} & = \frac{1}{\pi}\sum_{d< 2P} \sum_{\substack{N\in [X,X+Y]\\ N\in \mathcal{A}_{r,d}}}\sum_{n=1}^T\frac{\sqrt{4P^2X-r^2X^2}}{nd}\left(\frac{(r^2N^2-4P^2N)/d^2}{n}\right)\\
& \quad + O\left(P\sqrt{X}\sum_{d< 2P} \sum_{\substack{N\in [X,X+Y]\\ N\in \mathcal{A}_{r,d}}}\frac{P\sqrt{X}\log(P^2X)}{d^2T}+Y^2P^{1+\epsilon}X^{-\frac{1}{2}+\epsilon}+rX^{\frac{1}{2}+\epsilon}Y^{\frac{3}{2}}P^{\epsilon}\right)\\
&=\frac{1}{\pi}\sum_{d< 2P} \sum_{\substack{N\in [X,X+Y]\\ N\in \mathcal{A}_{r,d}}}\sum_{n=1}^T\frac{\sqrt{4P^2X-r^2X^2}}{nd}\left(\frac{(r^2N^2-4P^2N)/d^2}{n}\right)\\
& \quad + O\left(Y\frac{P^2X{\log(P^2X)}}{T}+Y^2P^{1+\epsilon}X^{-\frac{1}{2}+\epsilon}+rX^{\frac{1}{2}+\epsilon}Y^{\frac{3}{2}}P^{\epsilon}\right)\\
&=\frac{\sqrt{4P^2X-r^2X^2}}{\pi}\sum_{d< 2P}\sum_{n\le T}\frac{S_{n,d,r}}{nd} + O\left(\frac{YXP^2\log(P^2 X)}{T}+\frac{Y^2P^{1+\epsilon}}{X^{\frac{1}{2}-\epsilon}}+rX^{\frac{1}{2}+\epsilon}Y^{\frac{3}{2}}P^{\epsilon}\right),
\end{align*}}
where 
\[
S_{n,d,r} :=\sum_{\substack{N\in [X,X+Y]\\ N\in \mathcal{A}_{r,d}}}\mu^2(N)\left(\frac{N}{n}\right)\left(\frac{(r^2N-4P^2)/d^2}{n}\right) = \sum_{\substack{N\in [X,X+Y]\\ N\in \mathcal{A}_{r,d}}}\left(\frac{N}{n}\right)\left(\frac{(r^2N-4P^2)/d^2}{n}\right).
\]

The last equality follows from the fact that $\mu^2(N)=1$.
Indeed, $N\in \mathcal{A}_{r,d}$ forces $N$ to be always square-free. 
We can now state the main result of this section; its proof will occupy the remainder of this section.

\begin{proposition}
\label{our version of Nina P3.3}
Let $P\neq 2$ be a prime, let $r$ be a positive integer, and let $X>Y>0$ be such that $4P^2 > r^2(X+Y)$.
Then
\begin{equation*}
\begin{split}
   \sum_{\substack{N\in [X, X+Y] \\ P\nmid N \\ N \text{ sq-free}}} H_1(r^2 N^2 - 4P^2 N) &= \frac{Y\sqrt{4P^2 X - r^2X^2}}{\zeta(2)\pi}\prod_{p}\frac{p^4 - 2p^2 - p +1}{(p^2 - 1)^2}\prod_{p\mid r}\left( 1+ \frac{p^2}{p^4 - 2p^2 - p + 1}\right) \\
   & + O\left( \left(PXY\right)^\epsilon\left((YP^2 X)^{\frac{3}{5}} + \frac{Y^2 P}{\sqrt{X}}  + r\sqrt{X}Y^{\frac{3}{2}} + PXY^{\frac{5}{18}} + PY^{\frac{8}{9}}\sqrt{X}\right)\right).
\end{split}   
\end{equation*}
\end{proposition}

\begin{proof}
In Proposition~\ref{our version of Nina P3.5} we show that
\begin{equation*}
\begin{split}
\sum_{d \leq 2P} \sum_{n\leq T} \frac{S_{n,d,r}}{nd} &= \frac{Y}{\zeta(2)}\prod_{p}\frac{p^4 - 2p^2 - p +1}{(p^2 - 1)^2}\prod_{p\mid r}\left( 1+ \frac{p^2}{p^4 - 2p^2 - p + 1}\right) \\
& + O\left(\sqrt{X}Y^{\frac{5}{18}} + \sqrt{Y}T^{\frac{1}{4}+\epsilon} + (P X)^\epsilon Y^{\frac{8}{9}} \right).
\end{split}
\end{equation*}
The simplified expression of \eqref{to-be.truncated} combined with the aforementioned proposition implies that
\begin{equation*}
\begin{split}
   \sum_{\substack{N\in [X, X+Y] \\ P\nmid N \\ N \text{ sq-free}}} H_1(r^2 N^2 - 4P^2 N) &= \frac{Y\sqrt{4P^2 X - r^2X^2}}{\zeta(2)\pi}\prod_{p}\frac{p^4 - 2p^2 - p +1}{(p^2 - 1)^2}\prod_{p\mid r}\left( 1+ \frac{p^2}{p^4 - 2p^2 - p + 1}\right) \\
   &+  O\left(\frac{YP^2X\log(P^2 X)}{T}+\frac{Y^2P^{1+\epsilon}}{X^{\frac{1}{2}-\epsilon}}+rX^{\frac{1}{2}+\epsilon}Y^{\frac{3}{2}}P^{\epsilon}\right)\\
   &+ O\left(PXY^{\frac{5}{18}} + P\sqrt{X}\sqrt{Y}T^{\frac{1}{4}+\epsilon} + P^{1+\epsilon} X^{\frac{1}{2}+\epsilon} Y^{\frac{8}{9}} \right).
\end{split}   
\end{equation*}
Now, by choosing $T = (YP^2X)^{\frac25}$ gives an error term of
\begin{align*}
& \ O\left((YP^2X)^{\frac35}\log(P^2X)+\frac{Y^2P^{1+\epsilon}}{X^{\frac{1}{2}-\epsilon}}+rX^{\frac{1}{2}+\epsilon}Y^{\frac{3}{2}}P^{\epsilon}+PXY^{\frac{5}{18}} + (YP^2X)^{\frac35+\epsilon} + P^{1+\epsilon} X^{\frac{1}{2}+\epsilon} Y^{\frac{8}{9}}\right)\\
= & \ O\left(Y^{\frac{3}{5}+\epsilon}P^{\frac65 +\epsilon} X^{\frac35 +\epsilon} +\frac{Y^2P^{1+\epsilon}}{X^{\frac{1}{2}-\epsilon}}+rX^{\frac{1}{2}+\epsilon}Y^{\frac{3}{2}}P^{\epsilon}+PXY^{\frac{5}{18}} +  P^{1+\epsilon} X^{\frac{1}{2}+\epsilon} Y^{\frac{8}{9}}\right) \\
= & \ O\left((YPX)^\epsilon\left((YP^2 X)^{\frac{3}{5}} +\frac{Y^2P}{\sqrt{X}} + r\sqrt{X}Y^{\frac{3}{2}} + PXY^{\frac{5}{18}} +  P^{1} X^{\frac{1}{2}} Y^{\frac{8}{9}}\right)\right). \qedhere
\end{align*}
\end{proof}

\begin{corollary}
\label{our version of Nina P3.2}
Let $P\neq 2$ be a prime, let $r$ be a positive integer, and let $X>Y>0$ be such that $4P^2 > r^2(X+Y)$ for each $r<\frac{2P}{\sqrt{X}}$.
Then
{\allowdisplaybreaks
\begin{equation*}
\begin{split}
\frac{\zeta(2) \pi}{YX}  \sum_{r< \frac{2P}{\sqrt{X}}} \sum_{\substack{N\in [X, X+Y] \\ P\nmid N \\ N \text{ sq-free}}}  H_1(r^2 N^2 - 4P^2 N) = \sum_{r\le \frac{P}{\sqrt{X}}} \left(\sqrt{\frac{4P^2}{X} - r^2}\right)\prod_{p}\frac{p^4 - 2p^2 - p +1}{(p^2 - 1)^2}\prod_{p\mid r}\left( 1+ \frac{p^2}{p^4 - 2p^2 - p + 1}\right) & \\
 + O\left((PX)^\epsilon\left(\frac{P^{\frac{11}{5}}}{Y^{\frac25}X^{\frac{9}{10}}}+\frac{YP^2}{X^2}+\frac{P^2Y^{\frac12}}{X^{\frac32}}+\frac{P^2}{X^{\frac12}Y^{\frac{13}{18}}}+\frac{P^2}{XY^{\frac19}}\right)\right). &
\end{split}   
\end{equation*}
}
\end{corollary}

\begin{proof}
Rearranging the terms in Proposition~\ref{our version of Nina P3.3} gives

{\allowdisplaybreaks
\begin{align*}
\frac{\zeta(2)\pi}{XY}  \sum_{r< \frac{2P}{\sqrt{X}}}\sum_{\substack{N\in [X, X+Y] \\ P\nmid N \\ N \text{ sq-free}}}  H_1(r^2 N^2 - 4P^2 N) = \prod_p\frac{p^4 - 2p^2 - p +1}{(p^2 - 1)^2} \sum_{r < \frac{2P}{\sqrt{X}}}\left(\sqrt{\frac{4P^2}{X} - r^2}\right) \prod_{p\mid r}\left( 1+ \frac{p^2}{p^4 - 2p^2 - p + 1}\right) &\\
+ \sum_{r < \frac{2P}{\sqrt{X}}} O\left( \left(PXY\right)^\epsilon \left(\frac{(YP^2 X)^{\frac{3}{5}}}{XY} + \frac{Y^2 P}{XY\sqrt{X}}  + r\frac{\sqrt{X}Y^{\frac{3}{2}}}{XY} + \frac{PXY^{\frac{5}{18}}}{XY} + \frac{PY^{\frac{8}{9}}\sqrt{X}}{XY}\right)\right).&
\end{align*}
}

Since the main term has the right shape, we will only focus on the error term.

{\allowdisplaybreaks
\begin{align*}
    & \ \sum_{r < \frac{2P}{\sqrt{X}}} O\left( \left(PXY\right)^\epsilon \left(\frac{(YP^2 X)^{\frac{3}{5}}}{XY} + \frac{Y^2 P}{XY\sqrt{X}}  + r\frac{\sqrt{X}Y^{\frac{3}{2}}}{XY} + \frac{PXY^{\frac{5}{18}}}{XY} + \frac{PY^{\frac{8}{9}}\sqrt{X}}{XY}\right)\right)\\
    = & \ \sum_{r < \frac{2P}{\sqrt{X}}} O\left( \left( PXY\right)^\epsilon \left( \frac{P^{\frac{6}{5}}}{(XY)^{\frac25}} + \frac{PY}{X^{\frac32}}  + r\frac{\sqrt{Y}}{\sqrt{X}} + \frac{P}{Y^{\frac{13}{18}}} + \frac{P}{\sqrt{X}Y^{\frac19}}\right)\right)\\
    = & \  O\left( \left( PXY\right)^\epsilon \left( \frac{P^{1+\frac{6}{5}}}{\sqrt{X}(XY)^{\frac25}} + \frac{P^2Y}{X^2}  +  \frac{P^2}{\sqrt{X}Y^{\frac{13}{18}}} + \frac{P^2}{XY^{\frac19}} + \sum_{r < \frac{2P}{\sqrt{X}}} r\frac{\sqrt{Y}}{\sqrt{X}} \right)\right) \\
    = & \  O\left( \left( YPX\right)^\epsilon \left( \frac{P^{\frac{11}{5}}}{X^{\frac{9}{10}}Y^{\frac25}} + \frac{P^2Y}{X^2}  +   \frac{P^2 \sqrt{Y}}{X^{\frac32}} + \frac{P^2}{\sqrt{X}Y^{\frac{13}{18}}} + \frac{P^2}{XY^{\frac19}} \right)\right). \qedhere
\end{align*}
}
\end{proof}

Comparing the conditions from each of the error terms, we see that the most restrictive one arises from the last term.
It yields the relation 
$\frac{\delta_2}{2}+ \frac{\delta}{9} < \frac19$.
So a strict condition would be to choose\footnote{The condition that $\delta_2 < \delta$ will arise naturally from the calculations in \S\ref{section: the remaining terms}} $\delta_2 < \delta < \frac{2}{11}$.
A more relaxed bound on $\delta$ is mentioned in the statement of the main theorem.

The main task is now to estimate $S_{n,d,r}$.

\subsection{Some Preliminary Calculations}
Let $a$ be an integer such that $a\pmod{d^2}\in \mathcal{R}_{r,d}$. 
The character $\frac{(r^2x-4P^2)/d^2}{n}$ is a character modulo $fnd^2$ where $f=4$ if $n$ is even and $f=1$ otherwise.
Thus, we have
\[
\sum_{\substack{N\in [X,X+Y]\\ P\nmid N\\N\equiv a\pmod {d^2}}}\mu^2(N)\left(\frac{N}{n}\right)\left(\frac{(r^2 N-4P^2)/d^2}{n}\right) = \sum_{\substack{b \pmod {fd^2n}\\ a\equiv b \pmod{d^2}}}\left(\frac{b}{n}\right)\left(\frac{(r^2b-4P^2)/d^2}{n}\right)\sum_{\substack{N\in [X,X+Y]\\ \gcd(P,N)=1\\ N\equiv b\pmod{fnd^2}}}\mu^2(N).
\]

Note that $P^2>\frac{X}{4}$ and throughout our calculations we also have that $P>\frac{d}{2}$.
If $\gcd(P,fnd^2)>1$, then $P\mid n$ and $\left(\frac{N}{n}\right)=0$ for all $N$ satisfying $P\mid N$.
If $\gcd(P,fnd^2)=1$, then the number of elements in $[X,X+Y]$ that are divisible by $P$ and congruent to $b\pmod{fnd^2}$ are $O\left(\frac{X}{fd^2nP}\right)$.
Thus, deleting the condition $P\nmid N$ in the inner sum gives
\[
\sum_{\substack{b \pmod {fd^2n}\\a\equiv b \pmod{d^2}}}\left(\frac{b}{n}\right)\left(\frac{(r^2b-4P^2)/d^2}{n}\right) \left( \sum_{\substack{N\in [X,X+Y]\\ N\equiv b\pmod{fnd^2}}}\mu^2(N) + O\left(\frac{X}{fd^2nP}\right)\right).
\]
Now using a result of C.~Hooley (see also \cite[Lemma 6.4]{Zub23}) the above expression can be rewritten as
\[
\sum_{\substack{b \pmod {fd^2n}\\a\equiv b \pmod{d^2}}}\left(\frac{b}{n}\right)\left(\frac{(r^2b-4P^2)/d^2}{n}\right) \left( \frac{Y\eta(d^2n)}{\zeta(2)f\varphi(d^2n)} + O\left( \frac{\sqrt{X}}{d\sqrt{n}} + d^{1+\epsilon}n^{\frac{1}{2}+\epsilon}\right) + O\left(\frac{X}{fd^2nP}\right)\right)
\]
which can be simplified to
\[
\frac{Y\eta(d^2n)}{\zeta(2)f\varphi(d^2n)} \sum_{\substack{b \pmod {fd^2n}\\a\equiv b \pmod{d^2}}}\left(\frac{b}{n}\right)\left(\frac{(r^2b-4P^2)/d^2}{n}\right)+O\left(\sqrt{Xn}/d+d^{1+\epsilon}n^{\frac{3}{2}+\epsilon}+\frac{X}{Pd^2}\right).
\]

\begin{definition}
Define the functions 
\begin{align*}
    \theta_r(m) & =\sum_{a\pmod{m}}\left(\frac{a}{m}\right)\left(\frac{ar^2-4P^2}{m}\right)\\ \widetilde{\varphi}_{r,d}(g) & =\sum_{\substack{a \pmod{d^2g}\\ a\pmod{d^2}\in \mathcal{R}_{r,d}}}\left(\frac{a}{g}\right)\left(\frac{ar^2-4P^2}{g}\right).
\end{align*}
\end{definition}

\begin{lemma}
\label{our version of Nina L3.8}
Let $g=(d^\infty, n)$ and set $n'=n/g$.
Then
\[
\sum_{\substack{b \pmod {fd^2n}\\b \pmod{d^2}\in \mathcal{R}_{r,d}}} \left( \frac{b}{n} \right) \left(\frac{(r^2b-4P^2)/d^2}{n} \right) = f\widetilde{\varphi}_{r,d}(g)\theta_r(n').
\]
\end{lemma}

\begin{proof}
Assume first that $f\mid d^2$ (i.e either $f=1$ or $d$ is even).
Then $(d^\infty,fn)=fg$ and moreover, $2\mid g$ if $f=4$.
Here, we recall that $f=4$ precisely when $n$ is even.

\begin{align*}
& \sum_{\substack{b \pmod {fd^2n}\\b \pmod{d^2}\in \mathcal{R}_{r,d}}}\left(\frac{b}{n}\right)\left(\frac{(r^2b-4P^2)/d^2}{n}\right)=\sum_{a\in \mathcal{R}_{r,d}}\hspace{1mm} \sum_{\substack{b \pmod {fd^2n}\\ a\equiv b \pmod{d^2}}}\left(\frac{b}{n}\right)\left(\frac{(r^2b-4P^2)/d^2}{n}\right)\\
&=\left(\sum_{b \pmod{n'}}\left(\frac{b}{n'}\right)\left(\frac{(r^2b-4P^2)/d^2}{n'}\right)\right)\left(\sum_{a\in \mathcal{R}_{r,d}}\sum_{\substack{b\pmod{d^2gf}\\b\equiv a\pmod{d^2}}}\left(\frac{b}{gf}\right)\left(\frac{(r^2b-4P^2)/d^2}{gf}\right)\right)\\
&=\theta_r(n')\widetilde{\varphi}_{r,d}(fg).
\end{align*}
If $f=1$, this is the desired result. If $f=4$, then $g$ and $d$ are even and $\widetilde{\varphi}_{r,d}(fg)=f\widetilde{\varphi}_{r,d}(g)$. 

It remains to consider the case that $n$ is even and $d$ is odd. 
\begin{align*}
 & \sum_{\substack{b \pmod {fd^2n}\\b \pmod{d^2}\in \mathcal{R}_{r,d}}}\left(\frac{b}{n}\right)\left(\frac{(r^2b-4P^2)/d^2}{n}\right)=\sum_{a\in \mathcal{R}_{r,d}}\hspace{1mm} \sum_{\substack{b \pmod {fd^2n}\\a\equiv b \pmod{d^2}}}\left(\frac{b}{n}\right)\left(\frac{(r^2b-4P^2)/d^2}{n}\right)\\&=\left(\sum_{b \pmod{fn'}}\left(\frac{b}{fn'}\right)\left(\frac{(r^2b-4P^2)/d^2}{fn'}\right)\right)\left(\sum_{a\in \mathcal{R}_{r,d}}\sum_{\substack{b\pmod{d^2g}\\b\equiv a\pmod{d^2}}}\left(\frac{b}{g}\right)\left(\frac{(r^2b-4P^2)/d^2}{g}\right)\right)\\&=\widetilde{\varphi}_{r,d}(g)\theta_r(fn').\end{align*}
   As $n'$ and $f$ are even in this case, we obtain $\theta_r(fn')=f\theta_r(n')$ finishing the proof. 
\end{proof}

\begin{lemma}
\label{lem:theta-phi}
Let $(r,d)$ be an admissible pair.
Then 
\[
\widetilde{\varphi}_{r,d}(g)=
\begin{cases}
    \varphi(g)\delta_{g=\square} \quad & 2\nmid g, \ 4\nmid d\\
    2\varphi(g)\delta_{g=\square} \quad &2\nmid g, \ 4\mid d\\
     2\varphi(g)\delta_{g=\square} \quad & 2\mid g, \ 4\mid r, \ \gcd(d,r)=2\\
      2\varphi(g)\delta_{g=\square} \quad & 2\mid g, \ 4\mid d, \  \gcd(d,r)=2\\
      0\quad & \textup{else}.
    \end{cases}\]
Let $p$ be a prime coprime to $P$.
The function $\theta_r$ is a multiplicative function satisfying
\[
\theta_r(p^\alpha)=
\begin{cases} -p^{\alpha-1}\quad & 2\neq p, \ 2\nmid \alpha, \ p\nmid r\\
p^{\alpha-1}(p-2)\quad & 2\neq p, \ 2\mid \alpha, \ p\nmid r\\
0\quad & p\mid r, \ 2\neq p, \ 2\nmid \alpha \textup{ or } p=2, \ 2\mid r\\
p^{\alpha-1}(p-1)\quad & p\mid r, \ p\neq 2, \ 2\mid \alpha\\
(-1)^\alpha2^{\alpha-1} \quad & p=2, \ 2\nmid r.
\end{cases}
\]
\end{lemma}

\begin{proof}
The definition of the functions $\theta$ and $\varphi$ differ from the ones given in \cite{Zub23}.
Nevertheless, the proofs of Lemmas 6.2 and 6.3 in \emph{loc. cit.} still apply with minimal changes (substitute $P$ by $P^2$). 
\end{proof}

In nutshell, the calculation we have done so far yields
\begin{align*}
S_{n,d,r} & := \sum_{\substack{N\in [X,X+Y]\\ N\in \mathcal{A}_{r,d}}}\mu^2(N)\left(\frac{N}{n}\right)\left(\frac{(r^2N-4P^2)/d^2}{n}\right)\\
& = \frac{Y\eta(d^2n)}{\zeta(2)\varphi(d^2n)} \widetilde{\varphi}_{r,d}(g)\theta_r(n') + O\left(\sqrt{Xn}/d+d^{1+\epsilon}n^{\frac{3}{2}+\epsilon}+\frac{X}{Pd^2}\right).
\end{align*}
This is an approximation of $S_{n,d,r}$ in terms of multiplicative functions which are easier to manipulate.
In the remainder of this section we will prove the following result.

\begin{proposition}
\label{our version of Nina P3.5}
Let $d$ and $r$ be positive integers and let $X>Y>0$.
Let $P$ be a prime satisfying $4P^2 > r^2(X+Y)$.
Then for a cut-off parameter $T\gg Y$
\begin{equation*}
\begin{split}
\sum_{d \leq 2P} \sum_{n\leq T} \frac{S_{n,d,r}}{nd} &= \frac{Y}{\zeta(2)}\prod_{p}\frac{p^4 - 2p^2 - p +1}{(p^2 - 1)^2}\prod_{p\mid r}\left( 1+ \frac{p^2}{p^4 - 2p^2 - p + 1}\right) \\
& + O\left(\sqrt{X}Y^{\frac{5}{18}} + \sqrt{Y}T^{\frac{1}{4}+\epsilon} + (P X)^\epsilon Y^{\frac{8}{9}} \right).
\end{split}
\end{equation*}
\end{proposition}

\begin{proof}
The equality is obtained by combining the results obtained in Propositions~\ref{our analogue of Nina P3.6} and \ref{Nina P3.9}, and the calculation in \S\ref{section: large d and error terms}.
\end{proof}

\subsection{Small \texorpdfstring{$d$}{} and small \texorpdfstring{$n$}{}}

In this section, we estimate the sum of $\frac{S_{n,d,r}}{nd}$ in the range $n\leq Y^\sigma$ and $d\leq Y^\tau$ (where $\sigma, \tau$ are parameters in the interval $[0,1]$) by explicitly examining equidistribution of square-free numbers in residue classes modulo $n$.

\begin{lemma}
\label{our version of Nina L6.6}
Let $\sigma, \tau$ be parameters with $0\leq \sigma,\tau \leq 1$.
Let $d,n,r\geq 1$ be integers.
\[
\sum_{\substack{n\le Y^\sigma \\ d\le Y^\tau}}\frac{\eta(d^2n)\tilde{\varphi}_{n,d}(r)\theta_r(n')}{\varphi(d^2n)nd} = \prod_{p}\frac{p^4 - 2p^2 - p +1}{(p^2 - 1)^2}\prod_{p\mid r}\left( 1+ \frac{p^2}{p^4 - 2p^2 - p + 1}\right) + O\left(Y^{-2\tau}+Y^{-\frac{\sigma}{5}}\right).
\]
\end{lemma}

\begin{proof}
The proof follows as in \cite[Lemma~6.6]{Zub23}.
Note that we can apply the computations in \emph{loc. cit.} because of Lemma~\ref{lem:theta-phi}.
\end{proof}

For the ease of notation, set
\[
Bc(r) := \prod_{p}\frac{p^4 - 2p^2 - p +1}{(p^2 - 1)^2}\prod_{p\mid r}\left( 1+ \frac{p^2}{p^4 - 2p^2 - p + 1}\right).
\]

\begin{proposition}
\label{our analogue of Nina P3.6}
Let $\sigma, \tau$ be parameters with $0\leq \sigma,\tau \leq 1$.
Let $d,n, r\geq 1$ be integers such that $P\nmid n$, $4P^2 > r^2(X+Y)$, and $(r,d)$ is an admissible pair.
Then
\[
\sum_{\substack{n\le Y^\sigma \\ d\le Y^\tau}}\frac{S_{n,d,r}}{nd}=\frac{YBc(r)}{\zeta(2)} + O\left(\sqrt{X}Y^{\frac{\sigma}{2}}+Y^{\tau+\epsilon+\frac{3}{2}\sigma}+Y^{1-2\tau}+Y^{1-\frac{\sigma}{5}}\right).
\]
\end{proposition}

\begin{proof}
With notation introduced above, we have
{\allowdisplaybreaks
\begin{align*}
\sum_{\substack{n\le Y^\sigma \\ d\le Y^\tau}}\frac{S_{n,d,r}}{nd} & = \sum_{\substack{n\le Y^\sigma \\ d\le Y^\tau}} \left(\frac{Y\eta(d^2n)\tilde{\varphi}_{r,d}(g)\theta_r(n')}{\zeta(2)\varphi(d^2n)nd}+O\left(\frac{\sqrt{Xn}}{nd^2}+\frac{d^{1+\epsilon}n^{\frac{3}{2}+\epsilon}}{nd}+\frac{X}{Pd^3n}\right)\right)\\
& = \frac{Y}{\zeta(2)} \sum_{\substack{n\le Y^\sigma \\ d\le Y^\tau}} \left(\frac{\eta(d^2n)\tilde{\varphi}_{r,d}(g)\theta_r(n')}{\varphi(d^2n)nd} \right) + \sum_{\substack{n\le Y^\sigma \\ d\le Y^\tau}} O\left(\frac{\sqrt{X}}{d^2\sqrt{n}} + d^{\epsilon}n^{\frac{1}{2}+\epsilon} + \frac{X}{Pd^3n}\right)\\
& = \frac{Y}{\zeta(2)} \left(  Bc(r)+O\left(Y^{-2\tau}+Y^{-\frac{\sigma}{5}}\right) \right) + \sum_{\substack{n\le Y^\sigma \\ d\le Y^\tau}}  O\left(\frac{\sqrt{X}}{d^2\sqrt{n}} + d^{\epsilon}n^{\frac{1}{2}+\epsilon} + \frac{X}{Pd^3n}\right)  \text{ by Lemma~\ref{our version of Nina L6.6}} \\
& = \frac{YBc(r)}{\zeta(2)}+O\left(\sqrt{X}Y^{\frac{\sigma}{2}}+Y^{\tau+\epsilon+\frac{3}{2}\sigma}+Y^{1-2\tau}+Y^{1-\frac{\sigma}{5}}+\frac{X}{P}\log(Y^\sigma)\right).
\end{align*}
}
Recall that $P^2 > \frac{X}{4}$.
Equivalently, we have that $\frac{X}{P} < 2\sqrt{X}$.
The claim follows immediately.
\end{proof}

\subsection{Small \texorpdfstring{$d$}{} and large \texorpdfstring{$n$}{}}
The main result of this section is the following.
The proof of \cite[Proposition~3.9]{Zub23} goes through almost verbatim even in our setting.
We give a sketch of the proof to show that the shape of the error term is the same.

\begin{proposition}
\label{Nina P3.9}
Let $P$ be an odd prime and $d,r$ be positive integers such that $(d,r)$ is an admissible pair.
Let $X>Y>0$ satisfying $(X+Y)^2 < 4P^2$.
For parameters $\sigma, \tau$ between $0$ and $1$, 
\[
\sum_{\substack{n\in [Y^\sigma, T] \\ d< Y^\tau} }\frac{S_{n,d,r}}{nd} \ll \log T \log Y \sqrt{X} + Y^{1-\frac{\sigma}{2}+\epsilon} + \sqrt{Y}T^{\frac{1}{4}+\epsilon}\log Y \text{ as } X\rightarrow\infty.
\]
\end{proposition}

\begin{lemma}[{\cite[Lemma~3.12]{Zub23}}]
Let $Y\leq X$, let $n,w, D$ be positive integers and $q$ be any integer.
Let $a$ be an integer satisfying $wa \equiv q\pmod{D}$ and $\gcd(a,D)=1$.
Define the set
\[
\mathcal{P} = \left\{ p\mid n \text{ odd } : \ 2\nmid \val_p(n), \ p\nmid \gcd(q,n), \text{ and } \gcd(D,w,p)=1  \right\}.
\]
Then
\[
\sum_{\substack{N\in[X,X+Y] \\ N\equiv{a\pmod{D}}}} \mu^2(N)\left( \frac{N}{n}\right)\left( \frac{\frac{wN-q}{D}}{n}\right) \ll \sqrt{X} + \frac{Y}{D\prod_{p\in \mathcal{P}} \frac{\sqrt{p}}{2}} + \frac{\log Y}{\sqrt{D}}\left( \sqrt{\frac{nY}{\prod_{p\in \mathcal{P}} \frac{\sqrt{p}}{2}}}\right).
\]
\end{lemma}

\begin{proof}[Proof of Proposition~\ref{Nina P3.9}]
Choose $w=r^2$, $D=d^2$, $q=4P^2$ in the above lemma, and choose $a$ to be an integer such that $a\pmod{d^2}\in \mathcal{R}_{r,d}$.
If $Y\leq X$, $(d,r)$ is an admissible pair of positive integers, $n$ is any positive integer, and $P$ is a prime satisfying $4P^2 > r^2(X+Y)$ then the set
\[
\mathcal{P} = \left\{ p\mid n \ : \ 2\nmid \val_p(n), \ p\neq 2,P \right\}.
\]
This is precisely the set $\mathcal{P}_n$ in \cite[Corollary~3.13]{Zub23}.
Thus, the calculations in \cite[Proof of Proposition~3.9]{Zub23} go through without change.
We explain some of the calculations in more detail than in \cite{Zub23}
{\allowdisplaybreaks
\begin{align*}
S_{n,d,r} & = \sum_{\substack{N\in [X,X+Y]\\ P\nmid N \\ N\equiv a\pmod{d^2}}}\mu^2(N)\left(\frac{N}{n}\right)\left(\frac{(r^2N-4P^2)/d^2}{n}\right)\\
&\ll \sqrt{X} + \frac{Y}{d^2 \prod_{p\in \mathcal{P}} \frac{\sqrt{p}}{2}} + \frac{\log Y}{d}\left( \sqrt{\frac{nY}{\prod_{p\in \mathcal{P}} \frac{\sqrt{p}}{2}}}\right).
\end{align*}
}
Hence,

{\allowdisplaybreaks
\begin{align*}
\sum_{\substack{n\in [Y^\sigma,T] \\ d< Y^\tau}} \frac{S_{n,d,r}}{nd} &\ll \sum_{\substack{n\in [Y^\sigma,T] \\ d< Y^\tau}} \left( \frac{\sqrt{X}}{nd} + \frac{Y}{nd^3 \prod_{p\in \mathcal{P}} \frac{\sqrt{p}}{2}} + \frac{\log Y}{nd^2}\left( \sqrt{\frac{nY}{\prod_{p\in \mathcal{P}} \frac{\sqrt{p}}{2}}}\right)\right) \\
& = \sqrt{X} \sum_{\substack{n\in [Y^\sigma,T] \\ d< Y^\tau}} \frac{1}{nd} + Y\sum_{\substack{n\in [Y^\sigma,T] \\ d< Y^\tau}} \frac{1}{nd^3 \prod_{p\in \mathcal{P}} \frac{\sqrt{p}}{2}} + {\sqrt{Y}\log Y} \sum_{\substack{n\in [Y^\sigma,T] \\ d< Y^\tau}}\frac{1}{d^2 \sqrt{n \prod_{p\in \mathcal{P}} \frac{\sqrt{p}}{2}}} \\
&\ll \log T \log Y \sqrt{X} + Y\sum_{n\in [Y^\sigma,T]} \frac{1}{n \prod_{p\in \mathcal{P}} \frac{\sqrt{p}}{2}} + \sqrt{Y}{\log Y} \sum_{n\in [Y^\sigma,T]} \frac{1}{\sqrt{n \prod_{p\in \mathcal{P}} \frac{\sqrt{p}}{2}}}\\
&\ll \log T \log Y \sqrt{X} + Y\sum_{n \geq Y^\sigma} \frac{1}{n \prod_{p\in \mathcal{P}} \frac{\sqrt{p}}{2}} + \sqrt{Y}{\log Y} \sum_{n\leq T} \frac{1}{\sqrt{n \prod_{p\in \mathcal{P}} \frac{\sqrt{p}}{2}}}.
\end{align*}
}

\noindent We write $n=x^2 y 2^\alpha P^\beta$, where $y$ is square-free and $\gcd(x,2P)=\gcd(y,2P)=1$. 
Therefore $\mathcal{P} = \{ p \ : \ p\mid y\}$.
We now focus on the second term of the expression

{\allowdisplaybreaks
\begin{align*}
\sum_{n\geq Y^\sigma} \frac{1}{n \prod_{p\in \mathcal{P}} \frac{\sqrt{p}}{2}} & = 
\sum_{n\geq Y^\sigma} \frac{n}{n^2 \prod_{p\in \mathcal{P}} \frac{\sqrt{p}}{2}}\\
&\leq \sum_{n\geq Y^\sigma} \frac{x^2 y^{\frac{1}{2}+\epsilon} 2^\alpha P^\beta}{n^2} = \sum_{n\geq Y^\sigma} \frac{x^2 y^{\frac{1}{2}+\epsilon} 2^\alpha P^\beta}{x^4 y^2 2^{2\alpha} P^{2\beta}}\\
&\ll\sum_{\alpha,\beta} \sum_{x\in \mathbb{N}} \frac{1}{2^{\alpha} P^\beta x^2} \sum_{y> \frac{Y^\sigma}{x^2}}y^{\frac{-3}{2} + \epsilon} = \sum_{\alpha,\beta} \sum_{x\leq Y^{\frac{\sigma}{2}}} \frac{1}{2^{\alpha} P^\beta x^2}\sum_{y> Y^{\sigma/2}}y^{\frac{-3}{2} + \epsilon} + \sum_{\alpha,\beta}\sum_{x> Y^{\frac{\sigma}{2}}} \frac{1}{2^{\alpha} P^\beta x^2}\\
&\ll \sum_{x\leq Y^{\frac{\sigma}{2}}} \frac{1}{x^2} \frac{1}{y^{\frac{\sigma}{2} -\epsilon}} + \sum_{x> Y^{\frac{\sigma}{2}}} \frac{1}{x^2}\\
&\ll Y^{\frac{-\sigma}{2} + \epsilon}
\end{align*}
}

Next, we focus on the third term.

{\allowdisplaybreaks
\begin{align*}
\sum_{n\leq T} \frac{1}{\sqrt{n \prod_{p\in \mathcal{P}} \frac{\sqrt{p}}{2}}} & = \sum_{n\leq T} \frac{1}{n} \sqrt{\frac{n}{\prod_{p\in \mathcal{P}} \frac{\sqrt{p}}{2}}}\\
& \ll \sum_{\substack{\alpha,\beta: \\ 2^\alpha P^\beta \leq T}} \sum_{x\leq \sqrt{\frac{T}{2^\alpha P^\beta}}} \sum_{y \leq \frac{T}{2^\alpha P^\beta x^2}} \frac{1}{{2^\alpha P^\beta x^2 y}} \times \frac{2^{\frac{\alpha}{2}} P^{\frac{\beta}{2}} x y^{\frac{1}{2}+\epsilon}}{y^{\frac{1}{4}}} \\
& \ll \sum_{\substack{\alpha,\beta: \\ 2^\alpha P^\beta \leq T}} \sum_{x\leq \sqrt{\frac{T}{2^\alpha P^\beta}}} \sum_{y \leq \frac{T}{2^\alpha P^\beta x^2}} \frac{1}{{2^{\frac{\alpha}{2}} P^{\frac{\beta}{2}} x y^{\frac{3}{4} - \epsilon}}}\\
&\ll \sum_{x\leq \sqrt{T}} \frac{1}{x} \left( \frac{T}{x^2}\right)^{\frac{1}{4}+\epsilon} = \sum_{x\leq \sqrt{T}} \frac{T^{\frac{1}{4}+\epsilon}}{x^{\frac{3}{2} - \epsilon}} \\
&\ll T^{\frac{1}{4} - \epsilon}
\end{align*}}

Putting this together we get that as $X\rightarrow \infty$,
\[
\sum_{\substack{n\in [Y^\sigma,T] \\ d< Y^\tau}} \frac{S_{n,d,r}}{nd} \ll \log T \log Y \sqrt{X} + Y^{1-\frac{\sigma}{2}+\epsilon} + \sqrt{Y}T^{\frac{1}{4}+\epsilon}\log Y.
\]
\end{proof}

\subsection{Large \texorpdfstring{$d$}{} and error terms}
\label{section: large d and error terms}

The first main task of this section is to estimate the sum $ \frac{S_{n,d,r}}{nd}$ over all $n$ when $d \gg Y^\tau$.
We will observe that this also contributes (only) to the error term.

For $Y^\tau \ll d \ll P$, we have that

{\allowdisplaybreaks
\begin{align*}
\sum_{\substack{n\\ Y^\tau \ll d \ll P}} \frac{S_{n,d,r}}{nd} &= \sum_{Y^\tau \ll d \ll P}\frac{1}{d}\sum_n \frac{1}{n} \sum_{\substack{N\in [X,X+Y]\\ P\nmid N \\ N\equiv a\pmod{d^2}}}\mu^2(N)\left(\frac{N}{n}\right)\left(\frac{(r^2N-4P^2)/d^2}{n}\right)\\
&\leq  \sum_{Y^\tau \ll d \ll P}\frac{1}{d}\sum_n \frac{1}{n} \sum_{\substack{N\in [X,X+Y]\\ P\nmid N \\ N\equiv a\pmod{d^2}}}\left(\frac{N}{n}\right)\left(\frac{(r^2N-4P^2)/d^2}{n}\right) \text{ because } 0 \leq \mu^2(N) \leq 1.   \\
&\ll \sum_{Y^\tau \ll d \ll P} \frac{1}{d} \left(\frac{Y}{d^2} + 1 \right)\frac{(PX)^\epsilon}{d^{\epsilon}} \\
&\ll (P X)^\epsilon \left( Y^{1-2\tau} + \log P\right).
\end{align*}
}
We obtain the second  line by the trivial inequality  $\mu(N)^2\left(\frac{N}{n}\right)\left(\frac{(r^2N-4P^2)/d^2}{n}\right)\le 1 $.
The term $\left(\frac{Y}{d^2}+1\right)$ comes from the sum over $N$.

We now collect all the error terms.
The cumulative error terms from Propositions~\ref{our analogue of Nina P3.6} and \ref{Nina P3.9}, and the above calculations is
\begin{equation*}
\begin{split}
\left(\sqrt{X}Y^{\frac{\sigma}{2}}+Y^{\tau+\epsilon+\frac{3}{2}\sigma}+Y^{1-2\tau}+Y^{1-\frac{\sigma}{5}}\right) & + \left( \log T \log Y \sqrt{X} + Y^{1-\frac{\sigma}{2}+\epsilon} + \sqrt{Y}T^{\frac{1}{4}+\epsilon}\log Y \right)\\
& + (P X)^\epsilon \left( Y^{1-2\tau} + \log P\right)
\end{split}
\end{equation*}

Throughout this section we are assuming that $P^2 \ll X^{1+\delta_2}$ where $\delta_2$ is chosen as in the statement of the main theorem.
In particular, we may always choose $\delta_2 < \frac{2}{11}$.
Therefore, we may assume that $P^\epsilon \ll \sqrt{X}$ and choose the cut-off parameter $T\gg Y$, we can rewrite the cumulative error term as 
\[
\left(\sqrt{X}Y^{\frac{\sigma}{2}}+Y^{\tau+\epsilon+\frac{3}{2}\sigma}+Y^{1-2\tau}+Y^{1-\frac{\sigma}{5}}\right) + \left( \log T \log Y \sqrt{X} + Y^{1-\frac{\sigma}{2}+\epsilon} + \sqrt{Y}T^{\frac{1}{4}+\epsilon} \right) + (P X)^\epsilon Y^{1-2\tau}.
\]

Choosing $\tau = \frac{1}{18}$, $\sigma=\frac{5}{9}$ and assuming that $\log T \ll Y$, the above expression is bounded by
{\allowdisplaybreaks
\begin{align*}
\left(\sqrt{X}Y^{\frac{5}{18}}+Y^{\frac{8}{9}+\epsilon}+2Y^{\frac{8}{9}}\right) + \left( \log T \log Y \sqrt{X} + Y^{\frac{13}{18}+\epsilon} + \sqrt{Y}T^{\frac{1}{4}+\epsilon} \right) + (P X)^\epsilon Y^{\frac{8}{9}} &\\
\ll \sqrt{X}Y^{\frac{5}{18}} + \sqrt{Y}T^{\frac{1}{4}+\epsilon} + (P X)^\epsilon Y^{\frac{8}{9}}.
\end{align*}}

\section{The Remaining Terms}
\label{section: the remaining terms}
To complete the proof of the main theorem, we are still left to bound the following terms
\begin{itemize}
\item those in the range of $r$ not covered in Corollary~\ref{our version of Nina P3.2}
\item those involving levels $N$ when $P\mid N$
\item those involving the $P$ term in the trace formula.
\end{itemize}

First, we evaluate

{\allowdisplaybreaks
\begin{align*}
\frac{\zeta(2) \pi}{YX}   \sum_{\substack{N\in [X, X+Y] \\ P\nmid N \\ N \text{ sq-free}}} \sum_{ \frac{2P}{\sqrt{X+Y}} \leq r\leq  \frac{2P}{\sqrt{N}}} H_1(r^2 N^2 - 4P^2 N) & \ll \frac{Y}{XY} \left( P \left( \frac{1}{\sqrt{X}} - \frac{1}{\sqrt{X+Y}}\right)(P^2 X)^{\frac{1}{2} + \epsilon}\right) \\
& \ll \frac{1}{X}\left( \frac{P^{2+\epsilon}Y^{\frac12} X^{\frac12 + \epsilon}}{X} \right)\\
&\ll \frac{X^{1+\delta_2}X^{1-\frac{\delta}{2}}}{X^{2-\epsilon}} \\
& \ll X^{\delta_2 - \frac{\delta}{2} + \epsilon}.
\end{align*}}

This follows from the assumptions that $Y=(1+o(1))X^{1-\delta}$, 
that $P^2\ll X^{1+\delta_2}$, 
and that $P^\epsilon \ll \sqrt{X}$.
Since the main term is $\frac{P}{\sqrt{X}}$ which is of size $X^{\frac{\delta_2}{2}}$, it follows that we require $\delta_2<\delta$.

When $P\mid N$, we have noted before that
\[
\lambda_F(P) = \lambda_f(P^2) = \frac{1}{P}.
\]
In other words, at the ramified primes $P$, we notice that using the trivial bound

{\allowdisplaybreaks
\begin{align*}
\frac{\zeta(2) \pi}{XY}\sum_{\substack{N\in [X, X+Y] \\ P\mid N \\ N \text{ sq-free}}} \sum_{f\in H^{\textup{new}(N,2)}}P\lambda_f(P^2)\epsilon(f) &\leq \frac{\zeta(2) \pi}{XY} X \sum_{\substack{N\in [X, X+Y] \\ P\mid N \\ }} \abs{\epsilon(f)} \\
& \leq \frac{\zeta(2) \pi}{Y}\sum_{\substack{N\in [X, X+Y] \\ P\mid N }} 1 \\ 
&\leq  \zeta(2)\pi \frac{\#\{N \ : \ P\mid N \text{ and } N\in [X, X+Y]\}}{Y}  \\
&= O\left(\frac{\frac{Y}{P}+1}{Y}\right)=O\left(\frac{1}{P} + \frac{1}{Y}\right).
\end{align*}}


Finally, the note that
{\allowdisplaybreaks\begin{align*}
   \frac{\zeta(2)\pi}{XY} \sum_{\substack{N \in [X, X+Y] \\ N \text{ sq-free}} }(P+P^2) & = \frac{\zeta(2)\pi }{XY} (P+P^2)  \sum_{\substack{N \in [X, X+Y] \\ N \text{ sq-free}} } 1\\
    &=  \frac{\zeta(2)\pi }{XY} (P+P^2) \left(\sum_{N\leq X+Y} \mu^2 (N) - \sum_{N\leq X}\mu^2(N)\right)\\
    &= \frac{\zeta(2)\pi}{XY} (P+P^2) \left( \frac{Y}{\zeta(2)} + O(\sqrt{X+Y} - \sqrt{X})\right)\\
    &= \frac{\pi P^2}{X} + O\left( \frac{P}{X} + \frac{P^2}{XY^{\frac12}}\right).  
\end{align*}
}
\section {Arithmetic functions}
\label{Sec: Arithmetic Functions}

The purpose of this section is to record technical results required in the calculations performed above.

\begin{lemma}
\label{eta}
Let $P$ be an odd prime and $K$ be a cut-off parameter.
Define 
\[
A:=\prod_p\left(1+\frac{p}{(p+1)^2(p-1)}\right) \quad \textrm{ and } \quad \eta(m) := \prod_{p\mid m} \frac{p}{p+1}
\]
Then

\begin{align}
\label{eq: 5.1}
\sum_{\substack{m=1}}^K\frac{\eta(m)}{m^2} & = \prod_{p}\left(1+\frac{p}{(p+1)^2(p-1)}\right)+O\left(\frac{1}{K}\right)= A+O\left(\frac{1}{K}\right).\\
\label{eq: 5.2}
\sum_{\substack{m=1\\ \gcd(m,2)=1}}^K\frac{\eta(m)}{m^2} & = \prod_{p\neq 2}\left(1+\frac{p}{(p+1)^2(p-1)}\right)+O\left(\frac{1}{K}\right)=\frac{9}{11} A + O\left(\frac{1}{K}\right).\\
\label{eq: 5.3}
\sum_{\substack{m=1}}^K\frac{\eta(2m)}{m^2} & = \frac{8}{11}A+O\left(\frac{1}{K}\right).\\
\label{eq: 5.4}
\sum_{\substack{m=1\\ \gcd(m,2P)=1}}^K\frac{\eta(m)}{m^2} & = \prod_{p\neq 2,P}\left(1+\frac{p}{(p+1)^2(p-1)}\right)+O\left(\frac{1}{K}\right)=\frac{9}{11} A + O\left(\frac{1}{P^2} + \frac{1}{K}\right).\\
\label{eq: 5.5}
\sum_{\substack{m=1 \\ \gcd(m,P)=1}}^K\frac{\eta(2m)}{m^2} & = \frac{8}{11}A+O\left(\frac{1}{P^2} + \frac{1}{K}\right).
\end{align}
\end{lemma}

\begin{proof}
To see \eqref{eq: 5.1} we use the fact that
\[
\sum_{m=1}^\infty \frac{\eta(m)}{m^2} = \prod_p \left(\sum_{k=0}^{\infty}\frac{\eta(p^k)}{p^{2k}} \right) = \prod_p\left( 1+\frac{1}{p(p+1)} + \frac{1}{p(p+1)^3} +\ldots\right).
\]
The first equality follows by using the formula for the geometric series and truncating the sum at $m=K$.

For \eqref{eq: 5.2}, we note that the sum is over all $m$ that are odd.
Therefore, in the Euler product description we need to multiply over all primes except $p=2$.
Note that
\[
A = \left(1 + \frac{2}{9} \right) \prod_{p\geq 3}\left(1+\frac{p}{(p+1)^2(p-1)}\right) = \frac{11}{9} \prod_{p\geq 3}\left(1+\frac{p}{(p+1)^2(p-1)}\right).
\]
Rearranging the terms gives the desired result.

For \eqref{eq: 5.3}, we observe that
\[
\sum_{\substack{m=1}}^K\frac{\eta(2m)}{m^2} = 4\sum_{\substack{m=1}}^K\frac{\eta(2m)}{(2m)^2} = 4\sum_{\substack{m=1 \\ m \text{ even}}}^K\frac{\eta(m)}{m^2}.
\]
But, we also know that
\[
\sum_{\substack{m=1}}^K\frac{\eta(m)}{m^2} = \sum_{\substack{m=1 \\ m\text{ odd}}}^K\frac{\eta(m)}{m^2} + \sum_{\substack{m=1 \\ m\text{ even}}}^K\frac{\eta(m)}{m^2}.
\]
The desired relation can then be obtained from \eqref{eq: 5.1} and \eqref{eq: 5.2}. 

Note that \eqref{eq: 5.4} and \eqref{eq: 5.5} can be proven in the same way and are already stated in \cite[Lemma~6.5]{Zub23}. 
\end{proof}

\section{Main Result and Proof}

The main result we prove in this note is the following

\begin{theorem}
\label{main theorem}
Let $H^{\new}(N)$ be a Hecke basis for trivial character for weight 2 cusp forms for $\Gamma_0(N)$ with $f\in H^{\new}(N)$ normalized to have leading coefficient 1.
Let $\epsilon(f)$ be the root number of $f$ and let $a_f(p)$ be the $p$-th Fourier coefficient of $f$, and set $\lambda_f(p) = \frac{a_f(p)}{\sqrt{p}}$.
Let $P$ be a prime, and suppose that the parameters $P$, $X$, and $Y$ go to infinity.
Further suppose that $Y= (1+o(1))X^{1-\delta}$ and $P^2\ll X^{1+\delta_2}$
where $0<\delta_2 < \delta < \frac{9}{11}$ and $\frac{\delta}{9} + \frac{\delta_2}{2} < \frac19$.
Set $\delta'=\delta_2 - \frac{\delta}{2}$.
Then writing $y=\frac{P^2}{X}$,
\begin{align*}
\frac{\displaystyle\sum_{\substack{N\in [X, X+Y] \\ N \text{sq-free}}}\sum_{f\in H^{\new}(N)} P \lambda_{f}(P^2)\epsilon(f)}{\displaystyle\sum_{\substack{N\in [X, X+Y] \\ N \text{sq-free}}}\sum_{f\in H^{\new}(N)} 1} = & \frac{12}{\pi \prod_p\left( 1 - \frac{1}{p(p+1)}\right)}\left(A\sqrt{y} + B\sum_{r\leq \sqrt{y}} C(r)\left( \sqrt{4y-r^2}\right) - \pi y\right) \\
& + O_\epsilon\left( X^{\delta' + \epsilon} +\frac{1}{P}\right).
\end{align*}

\end{theorem}

\begin{proof}
In \cite[\S3.4]{Zub23}, it is proven that for $N$ square-free 
\[
\sum_{\substack{N\in [X, X+Y]\\ N \text{ sq-free}}} \sum_{f\in H^{\new}(N)} 1 = \frac{XY}{12 \zeta(2)}\prod_p \left(1- \frac{1}{p(p+1)}\right) + O\left(YX^\epsilon + X^{\frac85 + \epsilon} + Y^2 \right).
\]

Set $y=\frac{P^2}{X}$.
It will follow from Proposition~\ref{our version of Nina P3.1}, Corollary~\ref{our version of Nina P3.2}, and calculations in \S\ref{section: the remaining terms} that 

\begin{align*}
\frac{\zeta(2)\pi}{XY} \sum_{\substack{N\in [X, X+Y]\\ N \text{ sq-free}}} \sum_{f\in H^{\new}(N)} P\lambda_f(P^2) \epsilon(f) & = A\sqrt{y} + B \sum_{r\le \sqrt{y}} C(r)\left(\sqrt{4y - r^2}\right)\\
& - \pi y + O_\epsilon\left( X^{\delta'+\epsilon}+\frac{1}{P}\right).
\end{align*}
where $A= \prod_{p}\left(1 + \frac{p}{(p+1)^2 (p-1)}\right)$, $B= \prod_{p}\frac{p^4 - 2p^2 - p +1}{(p^2 - 1)^2}$, and $C(r) = \prod_{p\mid r}\left( 1+ \frac{p^2}{p^4 - 2p^2 - p + 1}\right)$.
Indeed, the error terms in Proposition~\ref{our version of Nina P3.1}, Corollary~\ref{our version of Nina P3.2}, and the calculations in \S\ref{section: the remaining terms} gives that the following inequalities are true (where we have set $\delta'=\delta_2-\frac{\delta}{2}$ and used the fact that
$\delta_2<\delta<\frac{2}{11}$ and $\delta_2+9\delta<2$ to obtain)
{\allowdisplaybreaks
\begin{align*}
    \frac{7}{12}\delta_2+\frac{5}{6}\delta+\left(\frac{7}{12}+\frac{1}{12}-\frac{5}{6}\right)=\delta'+\frac{8}{6}\delta-\frac{5}{12}\delta_2-\frac{1}{6}<\delta'\\
    \frac{\delta_2}{2}-\delta<\delta'\\
     \frac{11}{10}\delta_2+\frac{2}{5}\delta-\frac{1}{5}=\delta'+\frac{1}{10}\delta_2+\frac{9}{10}\delta-1/5<\delta'\\
     \delta_2- \delta<\delta'\\
     \delta_2+\frac{13}{18}\delta-\frac{2}{9}=\delta'+\frac{22}{18}\delta-\frac{2}{9}<\delta'\\
     \delta_2+\frac{1}{9}\delta-\frac{1}{9}=\delta'+\frac{11}{18}\delta-\frac{1}{9}<\delta'.
\end{align*}}
Thus, the cumulative error term is of size $O\left(X^{\delta'+\epsilon}+\frac{1}{P}\right)$.
\end{proof}

\bibliographystyle{amsalpha}
\bibliography{references}

\end{document}